\documentclass[12pt, reqno]{amsart}
\usepackage{amssymb, mathrsfs}
\usepackage{latexsym}
\usepackage{amsmath}
\usepackage{amsthm}
\usepackage{amsfonts}
\usepackage{dsfont, upgreek}
\usepackage{mathtools}
\usepackage{epsfig}
\usepackage{amscd}
\usepackage{graphicx}

\usepackage[final]{changes}

\usepackage{enumitem}
\usepackage{slashed}

\usepackage{tikz}
\usetikzlibrary{decorations.markings}
\usetikzlibrary{arrows.meta}

\usepackage[colorlinks=true,linkcolor=blue,citecolor=blue, urlcolor=blue]{hyperref}

\usepackage[all,cmtip]{xy}
\theoremstyle{plain}
\newtheorem{theorem}{Theorem}[section]
\newtheorem{proposition}[theorem]{Proposition}
\newtheorem{lemma}[theorem]{Lemma}
\newtheorem{corollary}[theorem]{Corollary}
\newtheorem{conjecture}[theorem]{Conjecture}

\theoremstyle{definition} 
\newtheorem{definition}[theorem]{Definition}

\theoremstyle{remark} 
\newtheorem{remark}[theorem]{Remark}
 
\newtheorem{question}[theorem]{Question}

\numberwithin{equation}{section}

\newcommand{\ind}{\textup{\,Ind}}

\newcommand{\tr}{\hspace{.3mm}\mathrm{tr}\hspace{.3mm}}

\newcommand{\ev}{\mathrm{ev}}

\newcommand{\dist}{\mathrm{dist}}

\newcommand{\spin}{\textup{spin}}

\newcommand{\SL}{\textup{SL}}
\newcommand{\sD}{\slashed{D}}

\begin{document}
\title[Approximations of delocalized eta invariants]{Approximations of delocalized eta invariants by their finite analogues}
\author{Jinmin Wang}
\address[Jinmin Wang]{School of Mathematical Sciences and Shanghai Center for Mathematical Sciences, Fudan University}
\email{wangjinmin@fudan.edu.cn}
\thanks{The first author is partially supported by NSFC 11420101001.}

\author{Zhizhang Xie}
\address[Zhizhang Xie]{Department of Mathematics, Texas A\&M University}
\email{xie@math.tamu.edu}
\thanks{The second author is partially supported by NSF 1800737.}

\author{Guoliang Yu}
\address[Guoliang Yu]{Department of Mathematics, Texas A\&M University}
\email{guoliangyu@math.tamu.edu}
\thanks{The third author is partially supported by NSF 1700021, NSF 1564398 and Simons Fellows Program.}

\maketitle
{\centering\footnotesize \em In memory of Vaughan Jones.\par}
\begin{abstract}
For a given self-adjoint first order elliptic differential operator on a closed smooth manifold,   we prove a list of results on when  the delocalized eta invariant associated to a regular covering space can be approximated by  the delocalized eta invariants associated to finite-sheeted covering spaces. One of our main results is the following. Suppose $M$ is a closed smooth spin manifold and $\widetilde M$ is a $\Gamma$-regular covering space of $M$. Let   $\langle \alpha \rangle$ be the conjugacy class of a non-identity element $\alpha\in \Gamma$. Suppose $\{\Gamma_i\}$ is a sequence of finite-index normal subgroups of $\Gamma$ that distinguishes $\langle \alpha \rangle$. Let  $\pi_{\Gamma_i}$ be the quotient map from $\Gamma$ to $\Gamma/\Gamma_i$ and $\langle \pi_{\Gamma_i}(\alpha) \rangle$ the conjugacy class of $\pi_{\Gamma_i}(\alpha)$ in $\Gamma/\Gamma_i$.   If the scalar curvature on $M$ is everywhere bounded below by a sufficiently large positive number,  then the  delocalized eta invariant for the Dirac operator of $\widetilde M$ at the conjugacy class $\langle \alpha \rangle$  is equal to the limit of the delocalized eta invariants for the Dirac operators  of $M_{\Gamma_i}$ at the conjugacy class  $\langle \pi_{\Gamma_i}(\alpha) \rangle$,  where $M_{\Gamma_i}= \widetilde M/\Gamma_i$ is the finite-sheeted covering space of $M$ determined by $\Gamma_i$.  In another main result of the paper, we prove that the limit of the delocalized eta invariants for the Dirac operators  of $M_{\Gamma_i}$ at the conjugacy class  $\langle \pi_{\Gamma_i}(\alpha) \rangle$ converges, under the assumption that the rational maximal Baum-Connes conjecture holds for $\Gamma$.
\end{abstract}

\section{Introduction}
The delocalized eta invariant for self-adjoint elliptic operators was first introduced by Lott \cite{Lott} as a natural extension of the classical eta invariant of Atiyah-Patodi-Singer \cite{A-P-S75a, A-P-S75b, A-P-S76}. It is a fundamental invariant in the studies of higher index theory on manifolds with boundary,  positive scalar curvature metrics on spin manfolds and rigidity problems in topology. More precisely, the delocalized eta invariant can be used to detect different connected components of the space of positive scalar curvature metrics on a given closed spin manifold \cite{MR1339924, LeichtnamPos}. Furthermore, it can be used to give an estimate of the connected components of the moduli space of positive scalar curvature metrics on a given closed spin manifold \cite{MR3590536}. Here the moduli space is obtained by taking the quotient of the space of positive scalar curvature metrics under the action of self-diffeomorphisms of the underlying manifold. As for applications to topology, the delocalized eta invariant can be applied to estimate the size of the structure group of a given closed topological manifold \cite{Weinberger:2016dq}. The delocalized eta invaraint is also closely related to the Baum-Connes conjecture. The second and third authors showed that if the Baum-Connes conjecture holds for a given group $\Gamma$, then\footnote{There is also an extra technical assumption that the conjugacy class $\langle \alpha \rangle$ used in the definition of the delocalized eta invariant is required to have polynomial growth.  }  the delocalized eta invariant associated to any regular $\Gamma$-covering space is an algebraic number \cite{Xie}. In particular, if a delocalized eta invariant is transcendental, then it would lead to  a counterexample to the Baum-Connes conjecture.    We refer the reader to \cite{Xie:2019aa}  for a more detailed survey of the delocalized eta invariant and its higher analogues.   


The delocalized eta invariant, despite being defined in terms of an explicit integral formula, is difficult to compute in general, due to its non-local nature.  The main purpose of this article is to study when the delocalized eta invariant associated to the universal covering of a space can be approximated by  the delocalized eta invariants associated to finite-sheeted coverings, where the latter are easier to compute. 

Let us first recall the definition of delocalized eta invariants. Let $M$ be a closed manifold and $D$ a self-adjoint elliptic differential operator on $M$. Suppose $\Gamma$ is a discrete group and $\widetilde M$ is a $\Gamma$-regular covering space of $M$. Denote by $\widetilde D$ the lift of $D$ from $M$ to $\widetilde M$. For any non-identity  element $\alpha \in
\Gamma$, the delocalized eta invariant $\eta_{\left\langle \alpha \right\rangle }(\widetilde D)$ of $\widetilde D$ at the conjugacy class $\langle \alpha \rangle$ is defined to be
\begin{equation}\label{eq:delocalizedeta}
\eta_{\left\langle \alpha \right\rangle }(\widetilde D)\coloneqq 
\frac{2}{\sqrt\pi}\int_{0}^\infty \sum_{\gamma \in\langle \alpha \rangle} \int_\mathcal F \tr(K_t(x,\gamma x))dxdt,
\end{equation}
where $K_t(x,y)$ is the Schwartz kernel of the operator $\widetilde D e^{-t^2\widetilde D^2}$ and $\mathcal F$ is a fundamental domain of $M_\Gamma$ under the action of $\Gamma$. 

\added{We point out that it is still open question whether the convergence of the integral in line \eqref{eq:delocalizedeta} holds in general. A list of cases where the convergence is known to hold is given right after  Definition \ref{def:deloceta}. In particular, if $\Gamma$ is finite, then the integral in line \eqref{eq:delocalizedeta} always converges. 
}

Now we consider finite-sheeted cover of $M$ given by finite-index normal subgroups of $\Gamma$. 
\begin{definition}\label{def:distinguishintro}
	Suppose that $\{\Gamma_i\}$ is a sequence of finite-index normal subgroups of $\Gamma$.
	For any non-trivial conjugacy class $\langle \alpha\rangle$ of $\Gamma$, we say that $\{\Gamma_i\}$ distinguishes $\langle \alpha\rangle$, if for any finite set $F$ in $\Gamma$ there exists $k\in\mathbb N_+$ such that 
	$$\forall \beta\in F,\ \beta\notin\langle \alpha\rangle\implies \pi_{
		\Gamma_i}(\beta)\notin\langle \pi_{\Gamma_i}(\alpha)\rangle
	$$
	for all $i\geqslant k$.
\end{definition}

Let $M_{\Gamma_i} = \widetilde M/\Gamma_i$ be the associated finite-sheeted covering space of $M$ and $D_{\Gamma_i}$ the lift of $D$ from $M$ to $M_{\Gamma_i}$. The delocalized eta invariant $\eta_{\langle \pi_{\Gamma_i}(\alpha) \rangle}(D_{\Gamma_i})$ of $D_{\Gamma_i}$ is defined similarly as in line $\eqref{eq:delocalizedeta}$, where $\pi_{\Gamma_i}$ is the canonical quotient map from $\Gamma$ to $\Gamma/\Gamma_i$.
Suppose $\{\Gamma_i\}$ distinguishes the conjugacy class $\langle \alpha \rangle$ of a non-identity element $\alpha\in \Gamma$, which is necessary for the following discussions.
  We prove a list of results that answer positively either one or both of the following questions. 
\begin{enumerate}
		\item  Does $\displaystyle  \lim_{i\to\infty}\eta_{\langle \pi_{\Gamma_i}(\alpha)\rangle }(D_{\Gamma_i})$
	exist?
	\item  If $\displaystyle  \lim_{i\to\infty}\eta_{\langle \pi_{\Gamma_i}(\alpha )\rangle }(D_{\Gamma_i})$ exists, is the limit equal to  $\eta_{\left\langle \alpha \right\rangle }(\widetilde D)$?
\end{enumerate}

For simplicity, we assume that  $M$ is a closed spin manifold equipped with a Riemannian metric of positive scalar curvature throughout the paper. Positive scalar curvature implies $\widetilde D$ has a spectral gap. In fact, the majority of results\footnote{such as Theorem $\ref{thm:A}$, Theorem $\ref{thm:B}$, Theorem $\ref{thm:conjcontrol}$ and Proposition $\ref{prop:sep}$ } in this paper can be proved in the same way under the assumption that   $\widetilde D$ has a spectral gap or  a sufficiently large spectral gap.

Here is one of the main results of our paper. 
\begin{theorem}\label{thm:A}
	With the above notation, assume that $\widetilde D$ is invertible and $\{\Gamma_i\}$ distinguishes the conjugacy class $\langle \alpha \rangle$ of a non-identity element $\alpha\in \Gamma$. If the maximal Baum-Connes assembly map for $\Gamma$ is rationally an isomorphism,  then the limit
	\[ \lim_{i\to \infty}\eta_{\langle\pi_{\Gamma_i}(\alpha)\rangle}(D_{\Gamma_i}) \] stabilizes, that is, $ \exists k>0$ such that   $\eta_{\langle\pi_{\Gamma_i}(\alpha)\rangle}(D_{\Gamma_i}) = \eta_{\langle\pi_{\Gamma_k}(\alpha)\rangle}(D_{\Gamma_k})
	$ for all $ i\geqslant k.$
\end{theorem}
Here we say $\{\Gamma_i\}$ distinguishes the conjugacy class $\langle \alpha \rangle$  if for any finite set $F$ in $\Gamma$, there exists $k\in\mathbb N_+$ such that 
	$$\forall \beta\in F,\ \beta\notin\langle \alpha\rangle\implies \pi_{
		\Gamma_i}(\beta)\notin\langle \pi_{\Gamma_i}(\alpha)\rangle
	$$
	for all $i\geqslant k$.

By a theorem of Higson and Kasparov \cite[Theorem 1.1]{MR1821144}, the maximal Baum-Connes assembly map is an isomorphism for all a-T-menable groups. We have the following immediate corollary. 
\begin{corollary}
With the above notation, assume that $\widetilde D$ is invertible and $\{\Gamma_i\}$ distinguishes the conjugacy class $\langle \alpha \rangle$ of a non-identity element $\alpha\in \Gamma$. If $\Gamma$ is a-T-menable, then the limit
\[ \lim_{i\to \infty}\eta_{\langle\pi_{\Gamma_i}(\alpha)\rangle}(D_{\Gamma_i}) \] stabilizes.
\end{corollary} 

Note that Theorem $\ref{thm:A}$ and its corollary above only addresses the first question, that is, only the convergence of  $\displaystyle  \lim_{i\to\infty}\eta_{\langle \pi_{\Gamma_i}(\alpha)\rangle }(D_{\Gamma_i})$. On the other hand, if in addition there exists a smooth dense subalgebra\footnote{A smooth dense subalgebra of $C^\ast_r(\Gamma)$  is a dense subalgebra of $C^\ast_r(\Gamma)$  that is closed under holomorphic functional calculus.} $\mathcal A$ of the reduced group $C^\ast$-algebra $C_r^\ast(\Gamma)$  of $\Gamma$ 
such that  $\mathbb C\Gamma\subset \mathcal A$ and the trace map\footnote{The trace map $ \tr_{\langle \alpha \rangle}$ is given by the formula: $ \sum_{\beta\in \Gamma}a_\beta \beta\mapsto \sum_{\beta\in\left\langle \alpha \right\rangle }a_\beta.$} $\tr_{\langle \alpha \rangle }\colon \mathbb C\Gamma \to \mathbb C$  extends continuously to a trace map $\tr_{\langle \alpha \rangle}\colon \mathcal A \to \mathbb C$, then we have
\[ \lim_{i\to \infty}\eta_{\langle\pi_{\Gamma_i}(\alpha)\rangle}(D_{\Gamma_i}) = \eta_{\left\langle \alpha \right\rangle }(\widetilde D). \]
See the discussion at the end of Section $\ref{sec:max}$ for more details.


Here is another main result of our paper. 
\begin{theorem}\label{thm:B}
With the above notation, suppose $\{\Gamma_i\}$ distinguishes the conjugacy class $\langle \alpha \rangle$ of a non-identity element $\alpha\in \Gamma$. If the spectral gap of $\widetilde D$ at zero is sufficiently large, then we have
\[ \lim_{i\to \infty}\eta_{\langle\pi_{\Gamma_i}(\alpha)\rangle}(D_{\Gamma_i}) = \eta_{\left\langle \alpha \right\rangle }(\widetilde D). \]
\end{theorem}	

Here ``sufficiently larger spectral gap"  means that the spectral gap of $\widetilde D$ at zero is greater than $\sigma_\Gamma$, where $\sigma_\Gamma$ is the constant 
given in Definition $\ref{def:cst}$. In particular, if the group $\Gamma$ has subexponential growth, then it follows from Definition $\ref{def:cst}$ that $\sigma_\Gamma = 0$. In this case, if $\widetilde D$ has a spectral gap, then it is automatically sufficiently large, hence the following immediate corollary. 

\begin{corollary}
	With the above notation, suppose $\{\Gamma_i\}$ distinguishes the conjugacy class $\langle \alpha \rangle$ of a non-identity element $\alpha\in \Gamma$. If $\Gamma$ has subexponential growth and $\widetilde D$ has a spectral gap at zero , then we have
	\[ \lim_{i\to \infty}\eta_{\langle\pi_{\Gamma_i}(\alpha)\rangle}(D_{\Gamma_i}) = \eta_{\left\langle \alpha \right\rangle }(\widetilde D). \]
\end{corollary}

 There are other variants of Theorem $\ref{thm:B}$ above. We refer the reader to Theorem $\ref{thm:conjcontrol}$ and Proposition $\ref{prop:sep}$ for details.

The paper is organized as follows. In Section \ref{sec:pre}, we review some basic facts about conjugacy separable groups and certain geometric $C^\ast$-algebras. In Section \ref{sec:deleta}, we review some basics of delocalized eta invariants. In Section \ref{sec:max}, we  prove one of main results, Theorem $\ref{thm:A}$, and discuss some of its consequences. In Section $\ref{sec:sc}$ and $\ref{sec:sep}$, we prove Theorem $\ref{thm:B}$ and its variants. 

We would like to thank the referees for helpful and constructive comments.

\section{Preliminaries}\label{sec:pre}
In this section, we review some basic facts about conjugacy separable groups and certain geometric $C^\ast$-algebras. 
\subsection{Conjugacy separable groups}\label{Sec: Conjugacy separable}
We will prove our main approximation results for a particular class of groups, called conjugacy separable groups. In this subsection, we review some basic properties of conjugacy separable groups. In the following, all groups are assumed to be finitely generated,  unless otherwise specified.
\begin{definition}\label{conjdis}
	Let $\Gamma$ be a finitely generated discrete group. We say that $\gamma \in \Gamma$ is conjugacy distinguished if for any $\beta\in \Gamma$ that is not conjugate to $\gamma$, there exists a finite-index normal subgroup $\Gamma'$ of $\Gamma$ such that the image of $\beta$ in $\Gamma/\Gamma'$ is not conjugate to $\gamma$.
\end{definition}
If every element in $\Gamma$ is conjugacy distinguished, then we say that $\Gamma$ is conjugacy separable. In other words, we have the following definition of conjugacy separability. \begin{definition}\label{def:conjsep}
	A finitely generated group $\Gamma$ is conjugacy separable if for any $\gamma_1, \gamma_2\in \Gamma$ that are not conjugate, there exists a finite-index normal subgroup $\Gamma'$ of $\Gamma$ such that the image of $\gamma_1$ and $\gamma_2$ in $\Gamma/\Gamma'$ are not conjugate.
\end{definition}

For any normal subgroup $\Gamma'$ of $\Gamma$, we denote by $\pi_{\Gamma'}$ the quotient map from $\Gamma$ to $\Gamma/\Gamma'$. 
\begin{definition}\label{def:distinguish}
	Suppose that $\{\Gamma_i\}$ is a sequence of finite-index normal subgroups of $\Gamma$.
	For any non-trivial conjugacy class $\langle \alpha\rangle$ of $\Gamma$, we say that $\{\Gamma_i\}$ distinguishes $\langle \alpha\rangle$, if for any finite set $F$ in $\Gamma$ there exists $k\in\mathbb N_+$ such that 
	$$\forall \beta\in F,\ \beta\notin\langle \alpha\rangle\implies \pi_{
		\Gamma_i}(\beta)\notin\langle \pi_{\Gamma_i}(\alpha)\rangle
	$$
	for all $i\geqslant k$.
\end{definition}
If $\alpha\in \Gamma$ is conjugacy distinguished, then such sequence always exists. More generally,  let $\mathfrak N$ be the net of all normal subgroups of $\Gamma$ with finite indices. If $\alpha\in G$ is conjugacy distinguished in the sense of Definition $\ref{conjdis}$, then $\mathfrak N$ distinguishes $\langle \alpha\rangle$, that is, for any finite set $F\subset \Gamma$, there exists a finite index normal subgroup $\Gamma_F$ of $\Gamma$ such that 
\[ \forall \beta\in F,\ \beta\notin\langle \gamma\rangle\implies \pi_{\Gamma'}(\beta)\notin\langle \pi_{\Gamma'}(\gamma)\rangle \]
for all $\Gamma'\in \mathfrak N$ with $\Gamma'\supseteq \Gamma_F$.

Let $\mathbb C\Gamma$ be the group algebra of $\Gamma$ and $\ell^1(\Gamma)$ be the $\ell^1$-completion of $\mathbb C\Gamma$. For any normal subgroup $\Gamma'$ of $\Gamma$,  the quotient map $\pi_{\Gamma'}\colon \Gamma
\to \Gamma/\Gamma'$ naturally induces an algebra homomorphism $\pi_{\Gamma'}\colon \mathbb C\Gamma\to \mathbb C(\Gamma/\Gamma')$, which extends to a Banach algebra homomorphism  $\pi_{\Gamma'} \colon \ell^1(\Gamma) \to \ell^1(\Gamma/\Gamma')$.

For any conjugacy class $\langle \gamma \rangle$ of $\Gamma$, let $\tr_{\langle \gamma \rangle }\colon \mathbb C\Gamma\to \mathbb C$ be the trace map defined by the formula:
\[ \sum_{\beta\in \Gamma}a_\beta \beta\mapsto \sum_{\beta\in\left\langle \gamma\right\rangle }a_\beta. \]
The following lemma is obvious. 
\begin{lemma}\label{lemma:limitoftrace}
	If $\langle \alpha\rangle$ is a non-trivial conjugacy class of $\Gamma$ and $\{\Gamma_i\}$ is a sequence of finite-index normal subgroups that  distinguishes $\langle \alpha\rangle$, then 
	\[ \lim_{i\to\infty}\tr_{\langle\pi_{\Gamma_i}(\alpha)\rangle }(\pi_{\Gamma_i}(f))=\tr_{\langle \alpha\rangle}(f) \]
	for all
	$f\in \ell^1(\Gamma)$. Moreover, if $f\in \mathbb C\Gamma$, then the limit on the left hand side stabilizes, that is,  
	\[ \exists k>0 \textup{ such that }  \tr_{\langle\pi_{\Gamma_i}(\alpha)\rangle }(\pi_{\Gamma_i}(f))=
	\tr_{\langle \alpha\rangle }(f), \textup{ for all } i\geqslant k.  \]
\end{lemma}

As we will mainly work with integral operators whose associated Schwartz kernels are smooth, let us fix some notation further and restate the above lemma in the context of integral operators. Let $M$ be a closed manifold and $\widetilde M$ be the universal covering space of $M$. Denote the fundamental group $\pi_1(M)$ of $M$ by $\Gamma$. Suppose  $T$ is a $\Gamma$-equivariant bounded smooth function on $\widetilde M\times\widetilde M$, that is, 
$$T(\gamma x,\gamma y)=T(x,y)$$
for all $x, y\in \widetilde M$ and $\gamma \in \Gamma$.
We say that $T$ has finite propagation if there exists a constant $d>0$ such that
$$\dist(x,y)>d\implies T(x,y)=0,$$
where $\dist(x, y)$ is the distance between $x$ and $y$ in $\widetilde M$. In this case, we define the propagation of $T$ to be the infimum of such $d$. 
\begin{definition}\label{def:L1}
A $\Gamma$-equivariant bounded function $T$ on $\widetilde M\times\widetilde M$ is said to be  $\ell^1$-summable if 
\[ \|T\|_{\ell^1} \coloneqq \sup_{x, y\in\mathcal F}\sum_{\gamma \in \Gamma}|T(x,\gamma y)|<\infty,\]
where $\mathcal F$ is a fundamental domain of $\widetilde M $ under the action of $\Gamma$. We shall call $\|T\|_{\ell^1}$ the $\ell^1$-norm of $T$ from now on. 
\end{definition} Clearly, every  $T$ with finite propagation is $\ell^1$-summmable. 

If a $\Gamma$-equivariant bounded smooth function $T\in C^\infty(\widetilde M \times \widetilde M)$ is $\ell^1$-summable, then  it defines a bounded operator on $L^2(\widetilde M)$ by the formula: 
\begin{equation}\label{eq:int}
f\mapsto \int_{\widetilde M}T(x,y)f(y)dy
\end{equation}
for all $ f\in L^2(\widetilde M)$. For notational simplicity, we shall still denote this operator by $T$.

Now suppose that $\Gamma'$ is a finite-index normal subgroup of $\Gamma$. Let $M_{\Gamma'} = \widetilde M/\Gamma'$ be the quotient space of $\widetilde M$ by  the action of $\Gamma'$. In particular, $M_{\Gamma'}$ is a finite-sheeted covering space of $M$ with the deck transformation group being $\Gamma/\Gamma'$. Let $\pi_{\Gamma'}$ be the quotient map from $\widetilde M$ to $M_{\Gamma'}$. Any $\Gamma$-equivariant bounded smooth function $T\in C^\infty(\widetilde M \times \widetilde M)$ that is $\ell^1$-summable naturally descends to  a smooth function $\pi_{\Gamma'}(T)$ on $M_{\Gamma'}\times M_{\Gamma'}$ by the formula: 
$$ \pi_{\Gamma'}(T)(\pi_{\Gamma'}(x),\pi_{\Gamma'}(y)):=\sum_{\gamma\in {\Gamma'}}T(x,\gamma y)$$
for all $(\pi_{\Gamma'}(x),\pi_{\Gamma'}(y))\in M_{\Gamma'}\times M_{\Gamma'}$. 
Clearly, $\pi_{\Gamma'}(T)$ is a $\Gamma/{\Gamma'}$-equivariant smooth function on $M_{\Gamma'}\times M_{\Gamma'}$ and,  similar to the formula in $\eqref{eq:int}$, defines a bounded operator on  $L^2(M_{\Gamma'})$.

For any non-trivial conjugacy class $\langle \alpha\rangle$ of $\Gamma$, we define the following trace map:
$$\tr_{\langle \alpha\rangle}(T)=\sum_{\gamma\in\langle \alpha\rangle}\int_{\mathcal F}T(x,\gamma x)dx,$$
for all $\Gamma$-equivariant $\ell^1$-summable smooth function $T\in C^\infty(\widetilde  M \times \widetilde M)$, where $\mathcal F$ is a fundamental domain of $\widetilde M $ under the action of $\Gamma$.  More generally, for each finite-index normal subgroup ${\Gamma'}$ of $\Gamma$, a similar trace map is defined for $\Gamma/{\Gamma'}$-equivariant smooth functions on $M_{\Gamma'}\times M_{\Gamma'}$.

With the above notation, Lemma $\ref{lemma:limitoftrace}$ can be restated as follows. 
 \begin{lemma}\label{lemma:limitkerneltrace}
 		Suppose $\langle \alpha\rangle$ is a non-trivial conjugacy class of $\Gamma$ and $\{{\Gamma}_i\}$ is a sequence of finite-index normal subgroups that distinguishes $\langle \alpha\rangle$.
 	 Let $T$ be a $\Gamma$-equivariant $\ell^1$-summable bounded smooth function on $\widetilde M\times\widetilde M$. Then we have
 	$$\lim_{i\to\infty}\tr_{\langle\pi_{{\Gamma}_i}(\alpha)\rangle }(\pi_{{\Gamma}_i}(T))=\tr_{\langle \alpha\rangle}(T).$$
 	Moreover, if $T$ has finite propagation, the the limit on the left hand side stabilizes, that is, $ \exists k>0$ such that   $\tr_{\langle\pi_{{\Gamma}_i}(\alpha)\rangle }(\pi_{{\Gamma}_i}(T))=
 	\tr_{\langle \alpha\rangle}(T)$,  for all $ i\geqslant k.$
 \end{lemma}

\subsection{Geometric $C^*$-algebras}
In this subsection, we review the definitions of some geometric $C^*$-algebras, cf.
\cite{Xiepos,Yulocalization} for more details.

Let $X$ be a proper metric space, i.e. every closed ball in $X$ is compact. An $X$-module is a separable Hilbert space equipped with a $*$-representation of $C_0(X)$. An $X$-module is called non-degenerated if the $*$-representation of $C_0(X)$ is non-degenerated. An $X$-module is called standard if no nonzero function in $C_0(X)$ acts as a compact operator.

In addition, we assume that a discrete group $\Gamma$ acts on $X$ properly and cocompactly by isometries. Assume $H_X$ is an $X$-module equipped with a covariant unitary representation of $\Gamma$. If we denote by $\varphi$ and $\pi$ the representations of $C_0(X)$ and $\Gamma$ respectively, this means
$$\pi(\gamma)(\varphi(f)v)=\varphi(\gamma^*f)(\pi(\gamma)v),$$
where $f\in C_0(X),\gamma\in \Gamma,v\in H_X$ and $\gamma^*f(x)=f(\gamma^{-1}x)$. In this case, we call $(H_X,\Gamma,\varphi)$ a covariant system.
\begin{definition}[\cite{YuChar}]
	A covariant system $(H_X,\Gamma,\varphi)$ is called admissible if
	
	\begin{enumerate}
		\item $H_X$ is a non-degenerate and standard $X$-module;
		\item for each $x\in X$, the stabilizer group $\Gamma_x$ acts regularly in the sense that the action is isomorphic to the action of $\Gamma_x$ on $l^2(\Gamma_x)\otimes H$ for some infinite dimensional Hilbert space $H$. Here $\Gamma_x$ acts on $l^2(\Gamma_x)$ by translations and acts on $H$ trivially.
	\end{enumerate}
	
\end{definition}

We remark that for each locally compact metric space $X$ with a proper, cocompact and isometric action of $\Gamma$, an admissible covariant system $(H_X,\Gamma,\varphi)$ always exists. In particular, if $\Gamma$ acts on $X$ freely, then the condition (2) above holds automatically.
\begin{definition}
	Let $(H_X,\Gamma,\varphi)$ be a covariant system and $T$ a $\Gamma$-equivariant bounded linear operator acting on $H_X$.
	\begin{itemize}
		\item The propagation of $T$ is defined to be 
		$$\sup\{d(x,y):(x,y)\in supp(T)\},$$
		where $supp(T)$ is the complement (in $X\times X$) of points $(x,y)\in X\times X$ for which there exists $f,g\in C_0(X)$ such that $gTf=0$ and $f(x)\ne 0,g(y)\ne 0$;
		\item $T$ is said to be locally compact if $fT$ and $Tf$ are compact for all $f\in C_0(X)$.
	\end{itemize}
\end{definition}
\begin{definition}
	Let $X$ be a locally compact metric space with a proper and cocompact isometric action of $\Gamma$. Let $(H_X,\Gamma,\varphi)$ be an admissible covariant system. We denote by $\mathbb C[X]^\Gamma$ the $*$-algebra of all $\Gamma$-equivariant locally compact bounded operators acting on $H_X$ with finite propagations. We define the equivariant Roe algebra $C^*(X)^\Gamma$ to be the completion of $\mathbb C[X]^\Gamma$ under the operator norm.
\end{definition}

Indeed, $C^*(X)^\Gamma$ is isomorphic to $C_r^*(\Gamma)\otimes \mathcal K$, the $C^*$-algebraic tensor product of the reduced group $C^*$-algebra of $\Gamma$ and the algebra of compact operators.
\begin{definition}
	We define the localization algebra $C^*_L(X)^\Gamma$ to be the $C^*$-algebra generated by all uniformly bounded and uniformly norm-continuous function $f:[0,\infty)\to C^\ast(X)^\Gamma$ such that the propagation of $f(t)$ goes to zero as $t$ goes to infinity.
	Define $C^*_{L,0}(X)^\Gamma$ to be the kernel of the evaluation map
	$$\ev:C^*_L(X)^\Gamma\to C^*(X),\ \ev(f)=f(0).$$
\end{definition}

Now let us also review the construction of higher rho invariants for invertible differential operators. For simplicity, let us focus on the odd dimensional case. Suppose $M$ is closed manifold of odd dimension. Let $M_\Gamma$ be the regular covering space of $M$ whose deck transformation group is $\Gamma$.  Suppose  $D$ is a self-adjoint elliptic differential operator on $M$ and $\widetilde D$ is the lift of $D$ to $M_\Gamma$. If $\widetilde D$ is invertible, then its higher rho invariant is defined as follows. 
\begin{definition}
	With the same notation as above, the higher rho invariant $\rho(\widetilde D)$ of an invertible operator $\widetilde D$ is defined to be
	\[ \rho(\widetilde D):=[e^{2\pi i\frac{\chi(\widetilde D/t)+1}{2}}]\in K_1(C^*_{L,0}(M_\Gamma)^\Gamma), \]
	where $\chi$ (called a normalizing function) is a continuous odd function such that $\lim_{x\to \pm \infty} \chi(x) = \pm 1$.
\end{definition} 

\added{By definition, the higher rho invariant $\rho(\widetilde D)$  is a uniformly norm-continuous function from $[0, \infty)$ to $C^\ast(X)^\Gamma$. It is a secondary invariant that serves as an obstruction for the higher index of $\widetilde D$ to be both trivial and local (i.e. having small propagation) at the same time, cf. \cite{MR4051922}. More precisely, for each fixed $t$, the unitary $e^{2\pi i\frac{\chi(\widetilde D/t)+1}{2}}$ is a representative of the higher index class of $\widetilde D$.  On one hand, since $\widetilde D$ is invertible, $\widetilde D$ has a spectral gap near zero. It follows that  $e^{2\pi i\frac{\chi(\widetilde D/t)+1}{2}}$ converges in norm to the trivial unitary $1$, as $t$ goes to zero.}   On the other hand, the propagation of $e^{2\pi i\frac{\chi(\widetilde D/t)+1}{2}}$ goes to zero (up to operators with small norm)\footnote{To be precise, one needs to use a normalizing function $\chi$ whose distributional Fourier transform has compact support, and furthermore approximate the function $e^{2\pi i x}$ by an appropriate polynomial. }, \added{ as $t$ goes to infinity. For an invertible $\widetilde D$, we can choose a representative of the higher index of $\widetilde D$ to be either trivial or local (i.e. having small propagation), but generally not both at the same time. In other words, the higher rho invariant measures the tension between the triviality and locality of the higher index of an invertible operator}

The above discussion has an obvious maximal analogue (cf. \cite[Lemma 3.4]{MR2431253}).
\begin{definition}
	For an operator $T\in\mathbb{C}[X]^\Gamma$, its \emph{maximal norm} is
	\[\|T\|_{\textnormal{max}}\coloneqq\sup_{\varphi}\left\{\|\varphi(T) \| : \varphi\colon \mathbb{C}[X]^\Gamma\rightarrow\mathcal{B}(H)\textrm{ is a  $*$-representation}\right\}.\]
	The maximal equivariant Roe algebra $C^*_{\max}(X)^\Gamma$ is defined to be the completion of $\mathbb{C}[X]^\Gamma$ with respect to $\|\cdot\|_{\textnormal{max}}$.
	Similarly, we define 
	\begin{enumerate}
		\item   the maximal localization algebra $C^*_{L, \max}(X)^\Gamma$ to be the $C^*$-algebra generated by all uniformly bounded and uniformly norm-continuous function $f:[0,\infty)\to C^\ast_{\max}(X)^\Gamma$ such that the propagation of $f(t)$ goes to zero as $t$ goes to infinity.
		\item and  $C^*_{L,0, \max}(X)^\Gamma$ to be the kernel of the evaluation map
		$$\ev:C^*_{L, \max}(X)^\Gamma\to C^*_{\max}(X),\ \ev(f)=f(0).$$		
	\end{enumerate}	
\end{definition}
Now suppose $M$ is a closed spin manifold. Assume that $M$ is endowed with a Riemannian  metric $g$ of positive scalar curvature. Let $M_\Gamma$ be the regular covering space of $M$ whose deck transformation group is $\Gamma$.  Suppose $D$ is the associated Dirac operator on $M$ and $\widetilde D$ is the lift of $D$ to $M_\Gamma$. In this case, we can define the maximal higher rho invariant of $\widetilde D$ as follows. 

\begin{definition}
	The maximal higher rho invariant $\rho_{\max}(\widetilde D)$ of $\widetilde D$ is defined to be
	\[ \rho_{\max}(\widetilde D):=[e^{2\pi i\frac{\chi(\widetilde D/t)+1}{2}}]\in K_1(C^*_{L,0, \max}(M_\Gamma)^\Gamma), \]
	Here $\chi$ is again a normalizing function, but the functional calculus for defining $\chi(t^{-1}\widetilde D)$ is performed under the maximal norm instead. See for example \cite[Section 3]{Guo:2019aa} for a discussion of such a functional calculus.
\end{definition}

\section{Delocalized eta invariants and their  approximations}\label{sec:deleta}
In this section, we review the definition of delocalized eta invariants and formulate the main question of this article.

We assume that  $M$ is a closed spin manifold equipped with a Riemannian metric of positive scalar curvature throughout the paper. Let $\Gamma$ be a finitely generated discrete group and $\widetilde M$ a $\Gamma$-regular covering space of $M$. Suppose $D$ is the associated Dirac operator on $M$ and $\widetilde D$ is the lift of $D$ to $\widetilde M$. 

Positive scalar curvature implies $\widetilde D$ has a spectral gap. In fact, the majority of results\footnote{such as Theorem $\ref{thm:A}$, Theorem $\ref{thm:B}$, Theorem $\ref{thm:conjcontrol}$ and Proposition $\ref{prop:sep}$ } in this paper also hold true under the assumption that   $\widetilde D$ has a spectral gap or  a sufficiently large spectral gap. For simplicity, we shall only  discuss the case where $M$ is a closed spin manifold equipped with a Riemannian metric of positive scalar curvature.

\begin{definition}[{\cite{Lott}}]\label{def:deloceta}
	For any conjugacy class $\left\langle\alpha\right\rangle$ of $\Gamma$, Lott's \emph{delocalized eta invariant} 	$\eta_{\left\langle\alpha\right\rangle }(\widetilde D)$ of $\widetilde D$ is defined to be
	\begin{equation}\label{delocalizedeta}
	\eta_{\left\langle\alpha\right\rangle }(\widetilde D)\coloneqq 
	\frac{2}{\sqrt\pi}\int_{0}^\infty \tr_{\left\langle\alpha\right\rangle }(\widetilde De^{-t^2\widetilde D^2})dt
	\end{equation}
	whenever the integral converges. Here
	$$\tr_{\left\langle\alpha\right\rangle }(\widetilde De^{-t^2\widetilde D^2})=
	\sum_{\gamma\in\langle\alpha\rangle}\int_\mathcal F \tr(k_t(x,\gamma x))dx,$$
where $k_t(x, y)$ is the corresponding Schwartz kernel of the operator $\widetilde De^{-t^2\widetilde D^2}$ and $\mathcal F$ is a fundamental domain of $\widetilde M$ under the action of $\Gamma$. 
\end{definition}

It is known that the integral formula $\eqref{delocalizedeta}$ for $\eta_{\left\langle\alpha\right\rangle }(\widetilde D)$ converges if $\widetilde D$ is invertible and any one of the following conditions is satisfied.
\begin{enumerate}
	\item The scalar curvature of $M$ is sufficiently large (see \cite[Definition 3.2]{CWXY} for the precise definition of ``sufficiently large").
	\item There exists a smooth dense subalgebra of $C^*_r(\Gamma)$ onto which the trace map $\tr_{\langle\alpha\rangle }$ extends continuously (cf. \cite[Section 4]{Lott}). For example, when $\Gamma$ is a  Gromov's hyperbolic group, Puschnigg's smooth dense subalgebra  \cite{Puschnigg} is such an subalgebra which admits a continuous extension of the trace map $\tr_{\langle\alpha\rangle}$ for all conjugacy classes $\langle h \rangle$.
	\item $\langle\alpha\rangle$ has subexponential growth (cf. \cite[Corollary 3.4]{CWXY}).
\end{enumerate}
In general, it is still an open question when the integral in $\eqref{delocalizedeta}$ converges for  invertible operators.

Now suppose that $\Gamma'$ is a finite-index normal subgroup of $\Gamma$. As before, let $M_{\Gamma'} = \widetilde M/\Gamma'$ be the associated finite-sheeted covering space of $M$. Similarly, let $D_{\Gamma'}$ be the lift of $D$ to $M_{\Gamma'}$, and define the delocalized eta invariant  $\eta_{\langle \pi_{\Gamma'}(\alpha)\rangle }(D_{\Gamma'})$ of $D_{\Gamma'}$ to be 
\begin{equation}\label{eq:fineta}
\eta_{\langle \pi_{\Gamma'}(\alpha)\rangle }(D_{\Gamma'})\coloneqq 
\frac{2}{\sqrt\pi}\int_{0}^\infty \tr_{\langle\pi_{\Gamma'}(\alpha)\rangle }(D_{\Gamma'}e^{-t^2 D_{\Gamma'}^2})dt,
\end{equation}
where $\alpha\in \Gamma$ and $\langle \pi_{\Gamma'}(\alpha)\rangle $ is conjugacy class of $\pi_{\Gamma'}(\alpha)$ in $\Gamma/\Gamma'$. As $M_{\Gamma'}$ is compact, it is not difficult to verify that the integral in $\eqref{eq:fineta}$ always converges absolutely. 

The above discussion naturally leads to the following questions.

\begin{question}\label{question}
Given a non-identity element $\alpha \in \Gamma$, suppose  $\{\Gamma_i\}$ is a sequence of finite-index normal subgroups that distinguishes the conjugacy class $\langle\alpha\rangle$. 
	\begin{enumerate}[label=(\Roman*), ref=(\Roman*)]
		\item \label{question1} When does $\displaystyle  \lim_{i\to\infty}\eta_{\langle \pi_{\Gamma_i}(\alpha)\rangle }(D_{\Gamma_i})$
		exist?
		\item \label{question2} If  $\eta_{\left\langle\alpha\right\rangle }(\widetilde D)$ is well-defined and $\displaystyle  \lim_{i\to\infty}\eta_{\langle \pi_{\Gamma_i}(\alpha)\rangle }(D_{\Gamma_i})$ exists, when do we have 
		\begin{equation}\label{eq:appr}
		\lim_{i\to\infty}\eta_{\langle \pi_{\Gamma_i}(\alpha)\rangle }(D_{\Gamma_i})=\eta_{\left\langle\alpha\right\rangle }(\widetilde D)?
		\end{equation}
	\end{enumerate} 
\end{question}

\section{Maximal higher rho invariants and their functoriality}\label{sec:max}
In this section, we use the functoriality of  higher rho invariants  to give some sufficient conditions under which the answer to part (I) of Question $\ref{question}$ is positive. 

Before we get into the technical details, here is a special case which showcases the main results of this section. 
\begin{proposition}\label{prop:atmenable}
With the same notation as in Question $\ref{question}$, if $\Gamma$ is a-T-menable and $\{\Gamma_i\}$ is a sequence of finite-index normal subgroups that distinguishes the conjugacy class $\langle\alpha\rangle$ for a non-identity element $\alpha\in \Gamma$, then the limit 
\[ \displaystyle  \lim_{i\to\infty}\eta_{\langle \pi_{\Gamma_i}(\alpha)\rangle }(D_{\Gamma_i}) \]
stabilizes, that is, $ \exists k>0$ such that   $\eta_{\langle\pi_{\Gamma_i}(\alpha)\rangle}(D_{\Gamma_i}) = \eta_{\langle\pi_{\Gamma_k}(\alpha)\rangle}(D_{\Gamma_k})
$,  for all $ i\geqslant k.$ In particular, $\displaystyle  \lim_{i\to\infty}\eta_{\langle \pi_{\Gamma_i}(\alpha)\rangle }(D_{\Gamma_i}) $ exists. 
\end{proposition} 
\begin{proof}
	This is a consequence of Theorem $\ref{prop:max}$ below and a theorem of Higson and Kasparov \cite[Theorem 1.1]{MR1821144}.
\end{proof}

Given a finitely presented discrete group $\Gamma$, let $\underline{E}\Gamma$ be the universal $\Gamma$-space for proper $\Gamma$-actions. The Baum-Connes conjecture \cite{PBAC88} can be stated as follows. 

\begin{conjecture}[Baum-Connes conjecture]\label{conj:bc}
The following map
$$\ev_*\colon  K_i(C^*_L(\underline{E}\Gamma)^\Gamma)
\to K_i(C^*(\underline{E}\Gamma)^\Gamma)$$
is an isomorphism. Here
$$K_i(C^*_L(\underline{E}\Gamma)^\Gamma)\coloneqq \varinjlim_{Y} K_i(C^*_L(Y)^\Gamma)$$
and 
$$K_i(C^*(\underline{E}\Gamma)^\Gamma)\coloneqq \varinjlim_{Y} K_i(C^*(Y)^\Gamma),$$
where the limit is taken over all $\Gamma$ cocompact spaces $Y$.
\end{conjecture}  Although this was not how the Baum-Connes conjecture was originally stated, the above formulation is equivalent to the original Baum-Connes conjecture, after one makes the following natural identifications: 
\[ K_i(C^*_L(\underline{E}\Gamma)^\Gamma)\cong K_i^\Gamma(\underline{E}\Gamma) \textup{ and } 
 K_i(C^*(\underline{E}\Gamma)^\Gamma)\cong K_i(C_r^*(\Gamma)). \]
Under this notation, we usually write the map \[ \ev_\ast\colon K_i(C^*_L(\underline{E}\Gamma)^\Gamma)\to K_i(C^*(\underline{E}\Gamma)^\Gamma\] as follows:
 \[ \mu\colon  K_i^\Gamma(\underline{E}\Gamma)\to K_i(C_r^*(\Gamma))  \]
and call it the Baum-Connes assembly map. 
Similarly,  there is a maximal version of the Baum-Connes assembly map: 
 \[ \mu_{\max} \colon  K_i^\Gamma(\underline{E}\Gamma)\to K_i(C_{\max}^*(\Gamma)). \]
The maximal Baum-Connes assembly map $\mu_{\max}$ is not an isomorphism in general. For example, $\mu_{\max}$ fails to be surjective for non-finite property (T) groups.

Before we discuss the functoriality of higher rho invariants, let us recall the functoriality of higher indices.  More precisely, let $D$ be a Dirac-type operator on a closed $n$-dimensional manifold $X$. Consider the following commutative diagram
\[ \xymatrixrowsep{1pc} \xymatrix{ & B\Gamma_1 \ar[dd]^{B\varphi} \\
	X \ar[dr]_{f_2}  \ar[ur]^{f_1} & \\
	& B\Gamma_2 }\]
where $f_1$, $f_2$ are continuous maps and $B\varphi$ is a continuous map from $B\Gamma_1$ to $B\Gamma_1$ induced by a group homomorphism $\varphi\colon \Gamma_1\to \Gamma_2$.  Let $X_{\Gamma_1}$ (resp. $X_{\Gamma_2}$) be the $\Gamma_1$ (resp. $\Gamma_2$) regular covering  space of $X$ induced by the map $f_1$ (resp. $f_2$), and $D_{X_{\Gamma_1}}$ (resp. $D_{X_{\Gamma_2}}$) be the lift of $D$ to $X_{\Gamma_1}$ (resp. $X_{\Gamma_2}$). We have the following functoriality of the higher indices: 
\[  \varphi_\ast(\ind_{\max}(D_{X_{\Gamma_1}}))  = \ind_{\max}(D_{X_{\Gamma_2}}) \textup{ in }  K_n(C_{\max}^\ast(\Gamma_2)),  \]
where $C_{\max}^\ast(\Gamma_i)$ is the maximal group $C^\ast$-algebra of $\Gamma_i$, the notation $\ind_{\max}$ stands for higher index in the maximal group $C^\ast$-algebra, and $\varphi_\ast\colon  K_n(C_{\max}^\ast(\Gamma_1))\to  K_n(C_{\max}^\ast(\Gamma_2))$ is the morphism naturally induced by $\varphi$.

Now let us consider the functoriality of higher rho invariants. Following the same notation from above, in addition, assume $X$ is  a closed spin manifold endowed with a Riemannian  metric of positive scalar curvature. In this case, the maximal higher rho invariants $\rho_{\max}(D_{X_{\Gamma_1}})$  of  $D_{X_{\Gamma_2}}$ and  $\rho_{\max}(D_{X_{\Gamma_1}})$ of $D_{X_{\Gamma_2}}$ are defined. Let  $E\Gamma_1$ (resp. $E\Gamma_2$) be universal $\Gamma_1$-space (resp. $\Gamma_2$-space ) for free $\Gamma_1$-actions (resp. $\Gamma_2$-actions). Denote by $\Phi$ the  equivariant map $ X_{\Gamma_1}\to X_{\Gamma_2}$ induced by $\varphi\colon \Gamma_1 \to \Gamma_2$, which in turn induces a morphism  \[ \Phi_\ast \colon  K_n(C_{L,0, \max}^\ast(X_{\Gamma_1})^{\Gamma_1}) \to K_n(C_{L,0, \max}^\ast(X_{\Gamma_2})^{\Gamma_2}).\] 
By \cite{guoxieyu}, 
the maximal higher rho invariants are functorial: 
\[ \Phi_\ast(\rho_{\max}(D_{X_{\Gamma_1}})) = \rho_{\max}(D_{X_{\Gamma_2}})\]
in $K_n(C_{L,0, \max}^\ast(X_{\Gamma_2})^{\Gamma_2})$.

Now suppose $M$ is an odd-dimensional closed spin manifold endowed with a positive scalar curvature metric and $\Gamma$ is a finitely generated discrete group. Let $\widetilde M$ be a $\Gamma$-regular covering space of $M$ and $\widetilde D$ be the Dirac operator lifted from $M$. For each finite-index normal subgroup $\Gamma'$ of $\Gamma$,  let $M_{\Gamma'} = \widetilde M/\Gamma'$ be the associated finite-sheeted covering space of $M$. Denote by $D_{\Gamma'}$ the Dirac opeartor on $M_{\Gamma'}$ lifted from $M$. 

\begin{theorem}\label{prop:max}
With the above notation, given a non-identity element $\alpha \in \Gamma$, suppose  $\{\Gamma_i\}$ is a sequence of finite-index normal subgroups that distinguishes the conjugacy class $\langle\alpha\rangle$.  If the maximal Baum-Connes assembly map for $\Gamma$ is rationally an isomorphism,  then
	\[ \lim_{i\to\infty}\eta_{\langle\pi_{\Gamma_i}(\alpha)\rangle}(D_{\Gamma_i}) \] stabilizes, that is, $ \exists k>0$ such that   $\eta_{\langle\pi_{\Gamma_i}(\alpha)\rangle}(D_{\Gamma_i}) = \eta_{\langle\pi_{\Gamma_k}(\alpha)\rangle}(D_{\Gamma_k})
	$,  for all $ i\geqslant k.$
\end{theorem}
\begin{proof}
We have the short exact sequence of $C^\ast$-algebras: 
	\[ 
	0\to C^*_{L,0, \max}(E\Gamma)^\Gamma \to C^*_{L, \max}(E\Gamma)^\Gamma \to C^*_{\max}(E\Gamma)^\Gamma \to 0\]
	which induces the following long exact sequence in $K$-theory:
	\begin{equation}\label{cd:longexact}
\scalebox{0.82}{ $\begin{CD}
	K_0(C^*_{L,0, \max}(E\Gamma)^\Gamma)\otimes\mathbb Q @>>>K_0(C^*_{L,\max }(E\Gamma)^\Gamma)\otimes\mathbb Q@>{\mu_0}>>K_0(C_{\max}^*(\Gamma))\otimes\mathbb Q\\
	@AAA@.  @VV\partial V \\
	K_1(C^*_{\max}(\Gamma))\otimes\mathbb Q @<{\mu_1}<<K_1(C^*_{L,\max }(E\Gamma)^\Gamma)\otimes\mathbb Q@<<<K_1(C^*_{L,0,\max }(E\Gamma)^\Gamma)\otimes\mathbb Q
	\end{CD} $}
	\end{equation} 
Note that $K_i(C^*_{L, \max}(E\Gamma)^\Gamma)$ is naturally isomorphic to  $K_i^\Gamma(E\Gamma)$. Similarly, we have $K_i(C^*_{L, \max}(\underline{E}\Gamma)^\Gamma)\cong K_i^\Gamma(\underline{E}\Gamma).$ The morphism $K_i^\Gamma(E\Gamma)\to K_i^\Gamma(\underline{E}\Gamma)$ induced by the  inclusion from $E\Gamma$ to $\underline{E}\Gamma$ is rationally injective (cf: \cite[Section 7]{BaumConnesHigson}). It follows that if the rational maximal Baum-Connes conjecture holds for $\Gamma$, that is, the maximal Baum-Connes assembly map $\mu_{\max}\colon  K_i(C_{L, \max}^\ast(\underline{E}\Gamma)^\Gamma)\otimes \mathbb 
	Q \to K_i(C_{\max}^\ast(\Gamma))\otimes \mathbb Q$ is an isomorphism, then the maps $\mu_i$ in the above commutative diagram are injective and the map $\partial $ is surjective. In particular, for the higher rho invariant $\rho(\widetilde D) $ of $\widetilde D$, there exists \[ [p]\in K_0(C^*_{\max}(E \Gamma)^\Gamma)\cong K_0(C^*_{\max}(\Gamma))\] such that $\partial [p]=\rho(\widetilde D)$ rationally, that is,  $\partial [p]=\lambda \cdot \rho(\widetilde D)$ for some $\lambda \in \mathbb Q$. 
	
	By the surjectivity of the Baum-Connes assembly map  
	\[ \mu_{\max}\colon  K_i(C_{L, \max}^\ast(\underline{E}\Gamma)^\Gamma)\otimes \mathbb 
	Q \to K_i(C_{\max}^\ast(\Gamma))\otimes \mathbb Q,\] we can assume $p$ is an idempotent  with finite propagation in $\mathcal K \otimes \mathbb C\Gamma$. 
\added{Indeed, let $[q]$ be an element in $ K_0(C_{L, \max}^\ast(\underline{E}\Gamma)^\Gamma)\otimes \mathbb 
		Q \cong K_0^\Gamma(\underline E\Gamma) \otimes \mathbb Q$  such that $\mu_{\max}([q])=[p]$.	It follows from the Baum--Douglas model of K-homology \cite{MR679698} that $[q]$ is the K-homology class of a twisted $\mathrm{spin^c}$ Dirac operator and $[p]$ is its higher index. More precisely,  
		there is an even-dimensional $\mathrm{spin^c}$ $\Gamma$-manifold $X$ together with a $\Gamma$-equivariant vector bundle $E$   such that rationally $[p]$ equals  the $\Gamma$-index of the twisted Dirac operator $\sD_E$ on $X$. For the convenience of the reader,  we shall review the construction of this index. Recall that  a function $\chi$ on $\mathbb R$ is called a normalizing function if $\chi \colon \mathbb R \to [-1, 1]$ is an odd continuous function such that  $\chi(x) >0$ when $x>0$,  and $
	\chi(x) \to \pm 1$ as $x \to \pm \infty$. Let $f$ be a smooth normalizing function whose distributional Fourier transform has compact support. }
Let us denote  $f(\sD_E)=\begin{pmatrix}
	0 & F_+\\ F_- & 0
	\end{pmatrix}$. Using the formula
	\[f(\sD) = \frac{1}{2\pi}\int_{\mathbb R} \widehat f(s) e^{isD} ds \] and the fact that $e^{isD}$ has propagation less than or equal to $|s|$, it follows that $f(\sD)$ has finite propagation, since the Fourier transform $\widehat f$ has compact support. Now we define 
	$$w=\begin{pmatrix}
	1&F_+\\0&1
	\end{pmatrix}\begin{pmatrix}
	1&0\\-F_-&1
	\end{pmatrix}\begin{pmatrix}
	1&F_+\\0&1
	\end{pmatrix}\begin{pmatrix}
	0&-1\\1&0
	\end{pmatrix}.$$
	Note that 
	\[ w^{-1} = \begin{pmatrix}
	0&1\\-1&0
	\end{pmatrix} \begin{pmatrix}
	1&-F_+\\0&1
	\end{pmatrix} \begin{pmatrix}
	1&0\\F_-&1
	\end{pmatrix} \begin{pmatrix}
	1&-F_+\\0&1
	\end{pmatrix}.\]
The higher index of $\sD$ is given as the following formal difference of idempotents:
	$$ \Big[w\begin{pmatrix}
	1&0\\0&0
	\end{pmatrix}w^{-1}\Big]-\Big[\begin{pmatrix}
	1&0\\0&0
	\end{pmatrix}\Big].$$
	By construction, we have
	\[ w\begin{pmatrix}
	1&0\\0&0
	\end{pmatrix}w^{-1}-\begin{pmatrix}
	1&0\\0&0
	\end{pmatrix} \in \mathcal S \otimes \mathbb C\Gamma \]
	where $\mathcal S $ is the algebra of trace class operators on a Hilbert space. In particular, we see that the element $[p]\in K_0(C^*_{\max}(\Gamma))$ from above can be (rationally)  represented by a formal difference of idempotents  with finite propagation.

Let $\Psi_i$ be the canonical quotient map from $ \widetilde M$ to $M_{\Gamma_i} = \widetilde M/{\Gamma_i}$ and \[ (\Psi_i)_\ast\colon  K_1(C_{L,0, \max}^\ast(\widetilde M)^\Gamma) \to K_1(C_{L,0, \max}^\ast(M_{\Gamma_i})^{\Gamma/\Gamma_i})\]
the corresponding morphism 
induced by $\Psi_i$. 		 By 
\cite[Theorem 1.1]{guoxieyu}, 
we have 
	\[   (\Psi_i)_\ast (\rho_{\max}(\widetilde D)) = \rho(D_{\Gamma_i}) \textup{ in } K_1(C_{L,0, \max}^\ast(M_{\Gamma_i})^{\Gamma/\Gamma_i}). \]
	By passing to the universal spaces, we have 
	\[   (\Psi_i)_\ast (\rho_{\max}(\widetilde D)) = \rho(D_{\Gamma_i}) \textup{ in } K_1(C^*_{L,0, \max}(E(\Gamma/\Gamma_i))^{\Gamma/\Gamma_i}). \]
Consider the following commutative diagram of long exact sequences\footnote{Since $\Gamma/\Gamma_i$ is finite, we have $C^*_{L,0, \max}(E(\Gamma/\Gamma_i))^{\Gamma/{\Gamma_i}} \cong C^*_{L,0}(E(\Gamma/\Gamma_i))^{\Gamma/{\Gamma_i}}$.}: 
		\begin{equation}\label{cd:longexact2}
	\begin{gathered} \scalebox{0.8}{ 
\xymatrixcolsep{1pc}	\xymatrix{
		  K_0(C^*_{L,\max }(E\Gamma)^\Gamma)\otimes\mathbb Q \ar[d] \ar[r]  & K_0(C_{\max}^*(\Gamma))\otimes\mathbb Q  \ar[r]^-{\partial} \ar[d]^{(\pi_{\Gamma_i})_\ast} & K_1(C^*_{L,0, \max}(E\Gamma)^\Gamma)\otimes\mathbb Q \ar[d]^{(\Psi_i)_\ast} \\
	 K_0(C^*_{L }(E(\Gamma/\Gamma_i))^{\Gamma/\Gamma_i})\otimes\mathbb Q \ar[r]  & K_0(C^*_r(\Gamma/\Gamma_i))\otimes\mathbb Q \ar[r]^-{\partial}  &  K_1(C^*_{L,0}(E(\Gamma/\Gamma_i))^{\Gamma/\Gamma_i})\otimes\mathbb Q   }
           }
       	\end{gathered}
		\end{equation} 
		where $(\pi_{\Gamma_i})_\ast\colon K_0(C_{\max}^*(\Gamma))\to K_0(C_r^*(\Gamma/\Gamma_i))$ is the natural morphism induced by the canonical quotient map $\pi_{\Gamma_i}\colon \Gamma\to \Gamma/\Gamma_i$.
Let us denote $(\pi_{\Gamma_i})_\ast(p)$ by $p_i$. It follows from the commutative diagram above that 
\begin{equation}\label{eq:aps}
\partial(p_i) = \rho(D_{\Gamma_i}). 
\end{equation} 
By \cite[Lemma 3.9 \& Theorem 4.3]{Xie}, for each $\Gamma_i$,  there exists a determinant map
  $$\tau_i\colon K_1(C^*_{L,0}(E(\Gamma/\Gamma_i))^{\Gamma/{\Gamma_i}}) \to \mathbb C$$ such that
\[ \frac{1}{2}\eta_{\langle \pi_{\Gamma_i}(\alpha)\rangle }(D_{\Gamma_i})  = -\tau_{i}(\rho(D_{\Gamma_i}))= \tr_{\langle\pi_{\Gamma_i}(\alpha)\rangle }(p_i). \]
 Since the idempotent $p$ has finite propagation, it follows from Lemma \ref{lemma:limitoftrace} that 
 $$\lim_{i \to \infty }\eta_{\langle \pi_{\Gamma_i}(\alpha)\rangle }(D_{\Gamma_i})=2\lim_{i \to \infty}\tr_{\langle\pi_{\Gamma_i}(\alpha)\rangle }(p_i)=2\tr_{\langle\alpha\rangle}(p),$$
 and  the limit stabilizes.

\end{proof}

\begin{remark}\ \\
	\begin{enumerate}
		\item In Theorem $\ref{prop:max}$ above, instead of the assumption that  $\{\Gamma_i\}$ is a sequence of finite-index normal subgroups that distinguishes the conjugacy class  $\langle\alpha\rangle$, we assume that $\{\Gamma_i\}$ is a decreasing sequence\footnote{We say $\{\Gamma_i\}$ is a decreasing sequence of finite-index normal subgroups of $\Gamma$ if $\Gamma_{i} \supseteq \Gamma_{i+1}$ for all $i$.} of finite-index normal subgroups of $\Gamma$. The same proof shows that if the maximal Baum-Connes assembly map for $\Gamma$ is rationally an isomorphism,  then
		\[ \lim_{i\to\infty}\eta_{\langle\pi_{\Gamma_i}(\alpha)\rangle}(D_{\Gamma_i}) \] stabilizes. On the other hand, to eventually relate the limit  $\lim_{i\to\infty}\eta_{\langle\pi_{\Gamma_i}(\alpha)\rangle}(D_{\Gamma_i}) $ to $\eta_{\langle \alpha \rangle}(\widetilde D)$, if the latter exists, one will likely have to assume the condition that $\{\Gamma_i\}$ distinguishes $\langle\alpha\rangle$. 
		\item Note that Theorem $\ref{prop:max}$ only answers part (I) of Question $\ref{question}$. Part (II) of Question $\ref{question}$ is still open, even under the assumption that the maximal Baum-Connes conjecture holds for $\Gamma$.
		\item Although Theorem $\ref{prop:max}$ assumes that the maximal Baum-Connes conjecture holds for $\Gamma$, it is clear from the proof  that it suffices to assume $\rho_{\max}(\widetilde D)$ is rationally in the image of the composition of the following maps: \[ K_0^\Gamma(\underline{E}\Gamma)\to K_0(C_{\max}^\ast(\Gamma)) \xrightarrow{\ \partial \  } K_1(C^*_{L,0,\max }(E\Gamma)^\Gamma). \]
	\end{enumerate}
  
\end{remark}
By a theorem of Higson and Kasparov \cite[Theorem 1.1]{MR1821144}, the maximal Baum-Connes conjecture holds for all a-T-menable groups. Together with Theorem $\ref{prop:max}$ above, this proves Proposition $\ref{prop:atmenable}$ at the beginning of the section.

As mentioned above, the maximal Baum-Connes assembly map $\mu_{\max}$ fails to be an isomorphism in general. For example, $\mu_{\max}$ fails to be surjective for non-finite property (T) groups. On the contrary, there is no counterexample to the  Baum-Connes conjecture, at the time of writing. In particular, the Baum-Connes conjecture is known to hold for all hyperbolic groups \cite{MR2874956, MR1914618}, many of which have property (T). For this reason,  we shall now investigate Question $\ref{question}$, in particular, the convergence of
	\[ \lim_{i\to\infty}\eta_{\langle\pi_{\Gamma_i}(\alpha)\rangle}(D_{\Gamma_i}) \]
when the group $\Gamma$ satisfies the Baum-Connes conjecture. 

One of the first difficulties we face is that reduced group $C^\ast$-algebras are not functorial with respect to group homomorphisms in general. As a result, the functoriality of higher rho invariants is, a priori, lost  in the reduced $C^\ast$-algebra setting.  Note that a key step (cf. Equation $\eqref{eq:aps}$) in the proof of Theorem $\ref{prop:max}$ is the existence of  a ``universal" idempotent $p\in \mathcal S\otimes \mathbb C\Gamma$ such that 
\[ \partial(p_i) = \rho(D_{\Gamma_i}), \]
where $p_i = (\pi_{\Gamma_i})_\ast(p)$. In the maximal setting, the existence of such a universal idempotent follows if the rational maximal Baum-Connes conjecture holds for $\Gamma$. \added{In the following, we shall discuss some geometric conditions that are sufficient for deriving an analogue of Theorem $\ref{prop:max}$ in the reduced setting. How these geometric conditions are related to  the (reduced) Baum-Connes conjecture will be explained in Appdendix $\ref{sec:stolz}$. }

Recall that $M$ is  a closed spin manifold equipped with a Riemannian metric $h$ of positive scalar curvature. Let $\varphi\colon M \to B\Gamma$ be the classifying map for the covering $\widetilde M \to M$, that is, the pullback of $E\Gamma$ by $\varphi$ is $\widetilde M$.  In the following, we denote by $\mathfrak B$ the Bott manifold, a simply connected spin manifold of dimension $8$ with $\widehat A(\mathfrak B) =1$. This manifold is not unique, but any choice will work for the following discussion. 

\begin{definition}\label{def:bound}
	We say  a multiple of $(M, \varphi, h)$  stably bounds with respect to $B\Gamma$ if there exists a compact spin manifold $W$ and a map $\Phi\colon W\to B\Gamma$ such that $\partial W = \bigsqcup_{i=1}^\ell M'$ and  $\Phi|_{\partial W} = \bigsqcup_{i=1}^\ell \varphi'$, where $(M', \varphi', h')$ is the direct product of $(M, \varphi, h)$ with finitely many copies of $\mathfrak B$ and $\bigsqcup_{i=1}^\ell M'$ is the disjoint union of $\ell$ copies of $M'$.
\end{definition}

\begin{definition}\label{def:properbound}
 Let $\tilde h$ be the metric on $\widetilde M$ lifted from $h$. We say a multiple of $(\widetilde M, \tilde h)$ positively stably  bounds with respect to $\underline{E}\Gamma$ if   there exists a spin cocompact $\Gamma$-manifold\footnote{Here a $\Gamma$-manifold is a Riemannian manifold equipped with a proper isometric action of $\Gamma$.} $\widetilde V$ equipped with a $\Gamma$-equivariant positive scalar curvature metric $g_{\widetilde V}$ such that $\partial \widetilde V = \bigsqcup_{i=1}^\ell \widetilde M'$ (as $\Gamma$-manifolds) and $g_{\widetilde V}$ has product structure near $\partial \widetilde V$, where $(\widetilde M', \widetilde h')$ is the direct product of $(\widetilde M, \widetilde h)$ with finitely many copies of $\mathfrak B$.
\end{definition}

\added{The following proposition is an analogue of Theorem $\ref{prop:max}$ in the reduced setting, under the  assumptions that a multiple of $(M, \varphi, h)$ stably bounds with respect to $B\Gamma$ and a multiple of 
$(\widetilde M, \tilde h)$ positively stably  bounds with respect to $\underline{E}\Gamma$. For example, if $M$ is  a lens space equipped with the metric inherited from the standard round metric on $S^n$ and $\Gamma = \pi_1(M)$, then both of these assumptions are satisfied. In general, the validity of these two assumptions is closely related to the reduced Baum-Connes conjecture and the Stolz conjecture on positive scalar curvature metrics. We refer the reader to Appendix $\ref{sec:stolz}$ for more details. }

\begin{proposition}\label{prop:red}
Let $M$ be a closed spin manifold equipped with a Riemannian metric $h$ of positive scalar curvature.	 Given a non-identity element $\alpha \in \Gamma$, suppose  $\{\Gamma_i\}$ is a sequence of finite-index normal subgroups of $\Gamma$ that distinguishes the conjugacy class $\langle 
	\alpha \rangle$. If  a multiple of $(M, \varphi, h)$ stably bounds with respect to $B\Gamma$ and a multiple of 
	$(\widetilde M, \tilde h)$ positively stably  bounds with respect to $\underline{E}\Gamma$, then 
	\[ \lim_{i\to\infty}\eta_{\langle\pi_{\Gamma_i}(\alpha)\rangle}(D_{\Gamma_i}) \] stabilizes, that is, $ \exists k>0$ such that   $\eta_{\langle\pi_{\Gamma_i}(\alpha)\rangle}(D_{\Gamma_i}) = \eta_{\langle\pi_{\Gamma_k}(\alpha)\rangle}(D_{\Gamma_k})
	$,  for all $ i\geqslant k.$
\end{proposition}
\begin{proof}
	For notational simplicity, let us assume $(M, \varphi, h)$ itself bounds with respect to $B\Gamma$,  that is, there exists a compact spin manifold $W$ and a map $\Phi\colon W\to B\Gamma$ such that $\partial W = M$ and  $\Phi|_{\partial W} = \varphi$. Similarly, let us assume  $(\widetilde M, \tilde h)$ itself positively stably  bounds with respect to $\underline{E}\Gamma$, that is,  there exists a cocompact $\Gamma$-spin manifold $\widetilde V$ with  $\partial \widetilde V = \bigsqcup_{i=1}^\ell \widetilde M'$ (as $\Gamma$-manifolds) and $\widetilde V$ is equipped with  a $\Gamma$-equivariant positive scalar curvature metric  that has product structure near $\partial \widetilde V$.  The general case can be proved in exactly the same way.

	Endow $W$ with a Riemannian metric $g$ which has product structure near $\partial W = M$ and whose restriction on $\partial W$ is the positive scalar curvature metric $h$. Let $\widetilde W$ be the covering space of $W$ induced by the map $\Phi\colon W\to B\Gamma$ and $\tilde g$ be the lift of $g$ from $W$ to $\widetilde W$. Due to the positive scalar curvature of $\tilde g$ near the boundary of $\widetilde W$,  the corresponding Dirac operator $D_{\widetilde W}$ on $\widetilde W$ with respect to the metric $\tilde g$ has a well-defined higher index $\ind(D_{\widetilde W}, \tilde g)$ in $KO_{n+1}(C_r^\ast(\Gamma;\mathbb R)). $ 
	
	Now for each normal subgroup $\Gamma_i$ of $\Gamma$, let  $M_{\Gamma_i} = \widetilde M/\Gamma_i$, $W_{\Gamma_i} = \widetilde W/\Gamma_i$ and $g_i$ be the lift of $g$ to $W_{\Gamma_i}$. Similarly, the corresponding Dirac operator $D_{W_{\Gamma_i}}$ on $W_{\Gamma_i}$ with respect to the metric $g_i$ has a well-defined higher index $\ind(D_{W_{\Gamma_i}}, g_i) $ in $KO_{n+1}(C_r^\ast(\Gamma/\Gamma_i;\mathbb R)). $ Moreover, we have 
	\[\partial (\ind(D_{W_{\Gamma_i}}, g_i) ) = \rho(D_{M_{\Gamma_i}}) \textup{ in }  KO_n(C_{L,0}^\ast(E(\Gamma/\Gamma_i); \mathbb R)^{\Gamma/\Gamma_i}), \] 
	cf. \cite[Theorem 1.14]{MR3286895}\cite[Theorem A]{Xiepos}.  
	
	By \cite[Lemma 3.9 \& Theorem 4.3]{Xie}, for each $\Gamma_i$,  there exists a determinant map
	$$\tau_i\colon K_1(C^*_{L,0, \max}(E(\Gamma/\Gamma_i))^{\Gamma/{\Gamma_i}}) \to \mathbb C$$ such that
	\[ \frac{1}{2}\eta_{\langle \pi_{\Gamma_i}(\alpha)\rangle }(D_{M_{\Gamma_i}})  = - \tau_{i}(\rho(D_{M_{\Gamma_i}}))= \tr_{\langle\pi_{\Gamma_i}(\alpha)\rangle }(\ind(D_{W_{\Gamma_i}}, g_i)). \]Therefore, to prove the proposition,  it suffices to show that there exists $[p]\in KO_{n+1}(C_{\max}^\ast(\Gamma;\mathbb R))$ such that $[p]$ is represented by a formal difference of idempotents in $\mathcal S\otimes \mathbb C\Gamma$ and 
	\[  (\pi_{\Gamma_i})_\ast([p]) = \ind(D_{W_{\Gamma_i}}, g_i)  \]
	for all $k$, where  $(\pi_{\Gamma_i})_\ast\colon C_{\max}^\ast(\Gamma; \mathbb R) \to  C_r^\ast(\Gamma/\Gamma_i; \mathbb R)  $ is the morphism induced by the quotient homomorphism $\pi_{\Gamma_i}\colon \Gamma \to \Gamma/\Gamma_i$. The existence of such a ``universal" $K$-theory element with finite propagation can be seen as follows.

	Let $Y$ be the spin $\Gamma$-manifold obtained by gluing $\widetilde V$ and $\widetilde W$ along their common boundary $\widetilde M$. Since the scalar curvature  on $\widetilde V$ is uniformly bounded below by a positive number, \added{it follows from the relative index theorem \cite{UB95,MR3122162} that}
	\[ \ind_{\max}(D_Y) = \ind_{\max}(D_{\widetilde W}, \tilde g)  \textup{ in } KO_{n+1}(C_{\max}^\ast(\Gamma;\mathbb R)). \]
	Let $p = \ind_{\max}(D_Y)$. By the discussion in the proof of Theorem $\ref{prop:max}$, the index class  $\ind_{\max}(D_Y)$ can be represented by a  formal difference of idempotents in $\mathcal S\otimes \mathbb C\Gamma$. On the other hand,  we have
	\[  (\pi_{\Gamma_i})_\ast(\ind_{\max}(D_{\widetilde W}, \tilde g))  = \ind(D_{W_{\Gamma_i}}, g_i) \]
	for all $i$. To summarize, we have 
	\[ (\pi_{\Gamma_i})_\ast([p]) = (\pi_{\Gamma_i})_\ast( \ind_{\max}(D_{\widetilde W}, \tilde g)) =  \ind(D_{W_{\Gamma_i}}, g_i).  \] This finishes the proof. 
	
\end{proof}

In Theorem $\ref{prop:max}$ and Proposition $\ref{prop:red}$, we have mainly focused on the  part (I) of Question $\ref{question}$. In the following, we shall try to answer part (II) of Question $\ref{question}$ in some special cases. Note that, a key ingredient of the proofs for Theorem $\ref{prop:max}$ and Proposition $\ref{prop:red}$ is the existence of a $K$-theory element\footnote{In the case of Proposition $\ref{prop:red}$, we map $KO$-theory to $K$-theory.} $[p_{\max}]\in K_{n+1}(C_{\max}^\ast(\Gamma)\otimes \mathcal K ) $ that is represented by a formal difference of idempotents in $\mathcal S\otimes \mathbb C\Gamma$ such that 
\[ \partial (p_{\max}) = \rho_{\max}(\widetilde D),\] where 
\[ \partial \colon  K_{n+1}(C_{\max}^\ast(\Gamma))\to KO_{n}(C_{L,0, \max}^\ast(E\Gamma)^\Gamma)\]
is the usual boundary map in the corresponding $K$-theory long exact sequence. We shall assume the existence of such a $K$-theory element $p_{\max}$ throughout the rest of the section. 

In addition, suppose there exists a smooth dense subalgebra $\mathcal A$ of $C_r^\ast(\Gamma)$  
such that  $\mathcal A\supset \mathbb C\Gamma$ and the trace map $\tr_{\langle \alpha \rangle }\colon \mathbb C\Gamma \to \mathbb C$  extends to a trace map $\mathcal A \to \mathbb C$. In this case, $\tr_{\langle \alpha \rangle }\colon  \mathcal A\to  \mathbb C$ induces a trace map 
\[ \tr_{\langle \alpha \rangle }\colon  K_0(C_r^\ast(\Gamma)) \cong K_0(\mathcal A)  \to \mathbb C\] 
and a determinant map (cf. \cite{Xie})
\[ \tau_{\alpha}\colon K_1(C^*_{L,0}(E\Gamma)^\Gamma)\to \mathbb C\]
such that the following diagram commutes: 
\[ \xymatrix{ K_0(C_r^\ast(\Gamma)) \ar[r]^-{\partial} \ar[d]_{-\tr_{\langle \alpha\rangle }} &  K_1(C^*_{L,0}(E\Gamma)^\Gamma) \ar[d]^{\tau_{\alpha}}\\
	\mathbb C \ar[r]^{=} &  \mathbb C} \]
Such a smooth dense subalgebra indeed exists  if $\langle \alpha \rangle$ has polynomial growth (cf. \cite{CM90}\cite{Xie}) or $\Gamma$ is word hyperbolic (cf. \cite{Puschnigg}\cite{CWXY}).

Note that the canonical morphism  
\[ K_1(C^*_{L,0, \max}(E\Gamma)^\Gamma)  \to K_1(C^*_{L,0}(E\Gamma)^\Gamma) \]
maps $\rho_{\max}(\widetilde D)$ to $\rho(\widetilde D)$. Let $p_r$ be the image of $p_{\max}$ under the canonical morphism $K_0(C_{\max}^\ast(\Gamma)) \to  K_0(C_r^\ast(\Gamma))$. The same argument from the proof of Theorem $\ref{prop:max}$ shows that $\partial(p_r) = \rho(\widetilde D)$ and
\[ \frac{1}{2}\eta_{\langle \alpha\rangle }(\widetilde D)  = - \tau_{\alpha}(\rho(\widetilde D))= \tr_{\langle\alpha\rangle }(p_r). \] 
Similarly, for each finite-index normal subgroup $\Gamma_i\subset \Gamma$, 	let 
\[ (\pi_{\Gamma_i})_\ast\colon K_0(C_{\max}^*(\Gamma))\to K_0(C_r^*(\Gamma/\Gamma_i))\] be the natural morphism induced by the  quotient map $\pi_{\Gamma_i}\colon \Gamma\to \Gamma/\Gamma_i$. Let us denote $ p_i \coloneqq  (\pi_{\Gamma_i})_\ast(p)$. We have  $\partial(p_i) = \rho(D_{\Gamma_i})$ and
\[ \frac{1}{2}\eta_{\langle \pi_{\Gamma_i}(\alpha)\rangle }(D_{\Gamma_i})  =  -\tau_{i}(\rho(D_{\Gamma_i}))= \tr_{\langle\pi_{\Gamma_i}(\alpha)\rangle }(p_i). \] 
where 
$$\tau_i\colon K_1(C^*_{L,0}(E(\Gamma/\Gamma_i))^{\Gamma/{\Gamma_i}}) \to \mathbb C$$ is a determinant map induced by the trace map $\tr_{\langle \pi_{\Gamma_i}(\alpha)\rangle}$, cf. \cite[Lemma 3.9 \& Theorem 4.3]{Xie}. Since $p_{\max}$ is a formal difference of idempotents in $\mathcal S\otimes \mathbb C\Gamma$, it follows that  the limit $ \displaystyle \lim_{i \to \infty } \tr_{\langle\pi_{\Gamma_i}(\alpha)\rangle }(p_i) $
stabilizes and is equal to $\tr_{\langle\alpha\rangle }(p_r)$. Thus  the limit \[ \lim_{i\to\infty}\eta_{\langle \pi_{\Gamma_i}(\alpha)\rangle }(D_{\Gamma_i}) \]
stabilizes and is equal to $\eta_{\left\langle \alpha\right\rangle }(\widetilde D)$.

In particular, as a consequence of the above discussion and Theorem $\ref{prop:max}$, we have the following theorem 

\begin{theorem}\label{thm:athyper}
	If $\Gamma$ is both a-T-menable and word hyperbolic\footnote{For example, if $\Gamma$  is a virtually free group, then it is both a-T-menable and word hyperbolic.}, then 
	\[ \lim_{i\to\infty}\eta_{\langle \pi_{\Gamma_i}(\alpha)\rangle }(D_{\Gamma_i}) = \eta_{\left\langle \alpha\right\rangle }(\widetilde D).  \]
\end{theorem}

\section{Scalar curvature and $\ell^1$-summability}\label{sec:sc}
In this section, we show that the answers to both part (I) and part (II) of Question $\ref{question}$ are positive when the scalar curvature of the given spin manifold $M$ is bounded below by a sufficiently large positive number.

Throughout this section, assume  $M$ is an odd-dimensional closed spin manifold endowed with a positive scalar curvature metric and $\Gamma$ is a finitely generated discrete group. Let $\widetilde M$ be a regular $\Gamma$-covering space of $M$ and $\widetilde D$ be the Dirac operator lifted from $M$. For each finite-index normal subgroup $\Gamma'$ of $\Gamma$,  let $M_{\Gamma'} = \widetilde M/\Gamma'$ be the associated finite-sheeted covering space of $M$. Denote by $D_{\Gamma'}$ the Dirac opeartor on $M_{\Gamma'}$ lifted from $M$. 

Let $S$ be a symmetric finite generating set of $\Gamma$ and $\ell$ be the  associated word length function on $\Gamma$. There exist $C>0$ and $B>0$ such that
\begin{equation}\label{eq:growthofGamma}
\#\{\gamma\in \Gamma :\ \ell(g)\leqslant n\}\leqslant Ce^{B\cdot n}.
\end{equation}
for all $n\geq 0$. Let $K_\Gamma$ be the infimum of all such numbers $B$.

Furthermore, there exist $\theta_0,\theta_1,c_0,c_1>0$ such that 
\begin{equation}\label{eq:equivwithlength}
\theta_0\cdot \ell(\beta)-c_0\leqslant \dist(x,\beta x)\leqslant \theta_1\cdot \ell(\beta)+c_1
\end{equation} 
for all $x\in \mathcal F$ and $\beta\in \Gamma$, where $\mathcal F$ is a fundamental domain of $\widetilde M$ under the action of $\Gamma$. 
In particular, we may define $\theta_0$ as follows:
\begin{equation}\label{eq:distortion}
\theta_0=\liminf_{\ell(\beta)\to\infty}\left(\inf_{x\in\mathcal F}\frac{\dist(x,\beta x)}{\ell(\beta)}\right).
\end{equation}
\begin{definition}\label{def:cst}
	With the above notation, let us define 
	\[ \sigma_\Gamma \coloneqq \frac{2 K_{\Gamma}}{\theta_0}. \]
\end{definition}

The following theorem answers both part (I) and part (II) of  Question $\ref{question}$ positively, under the condition that the spectral gap of $\widetilde D$ at zero is sufficiently large. 

\begin{theorem}\label{main}
	With the same notation as above,  suppose $\{\Gamma_i\}$ is a sequence of finite-index normal subgroups that distinguishes the conjugacy class $\langle \alpha\rangle$ of a non-identity element $\alpha\in \Gamma$.  If the spectral gap of $\widetilde D$ at zero is greater than $\sigma_\Gamma$, then
\[ 	\lim_{i\to\infty}\eta_{\langle \pi_{\Gamma_i}(\alpha)\rangle }(D_{\Gamma_i})=\eta_{\left\langle \alpha\right\rangle }(\widetilde D). \]
\end{theorem}	
\begin{proof}
	It suffices to find a function of $t$ that is a dominating function for all the following functions:
	\[ \tr_{\left\langle \alpha \right\rangle }(\widetilde De^{-t^2\widetilde D^2})=
	\sum_{\gamma\in\langle \alpha \rangle}\int_\mathcal F \tr( K_t(x,\gamma x) )dx\]
	and 
	\[ \tr_{\langle\pi_{\Gamma_i}(\alpha)\rangle }(D_{\Gamma_i}e^{-t^2D^2_{\Gamma_i}})
	=\sum_{\omega \in \langle\pi_{\Gamma_i}(\alpha )\rangle}\int_{\mathcal F}\tr((K_i)_t(x,\omega x))dx.
	\] 
	and show that $\tr_{\langle\pi_{\Gamma_i}(\alpha)\rangle }(D_{\Gamma_i}e^{-t^2D^2_{\Gamma_i}})$ converges to $\tr_{\left\langle \alpha \right\rangle }(\widetilde De^{-t^2\widetilde D^2})$, as $i\to\infty$, for each $t$.  Indeed, the theorem then follows by the dominated convergence theorem. 
	
	 Recall that $K_t(x,y)$ (resp. $(K_i)_t(x,y)$) is the Schwartz kernel of $\widetilde De^{-t^2\widetilde D^2}$ (resp. $D_{\Gamma_i} e^{-t^2 D_{\Gamma_i}^2}$). We have the following estimates (cf. \cite[Section 3]{CWXY}). 
	\begin{enumerate}
		\item By \cite[Lemma 3.8]{CWXY},  for  any $\mu>1$ and $r>0$,  there exists a constant $c_{\mu, r}>0$ such that 
		\begin{equation}\label{eq:bound}
		\|K_t(x,y)\|\leqslant c_{\mu, r}\cdot F_t\left(\frac{\dist(x, y)}{\mu}\right), 
		\end{equation}
		for $\forall x,y\in \widetilde M$ with  $\dist(x,y)>r$. Here  $\|K_t(x, y)\|$ is the operator norm of the matrix $K_t(x, y)$, and the function $F_t$ is defined by 
		\begin{equation*}
		F_t(s)\coloneqq \sup_{n\leqslant \frac{3}{2}\dim M+3}\int_{|\xi|>s}\left|\frac{d^n}{d\xi^n}\widehat f_t(\xi)\right|d\xi, 
		\end{equation*}
		where $\widehat f_t$ is the Fourier transform of $f_t(x) = xe^{-t^2x^2}$. It follows that for $\mu>1$ and $r>0$, there exist $c_{\mu, r}>0, n_1>0$ and $ m_1>0$ such that
		\begin{equation}\label{eq:smallt}
		\|K_t(x,y)\|\leqslant c_{\mu, r}\frac{(1+\dist(x,y))^{n_1}}{t^{m_1}}\exp\left( \frac{-\dist(x,y)^2}{4\mu t^2}\right),
		\end{equation}
		for all $t>0$ and for all $ x,y\in \widetilde M$ with  $\dist(x,y)>r$. 
		\item By \cite[Lemma 3.5]{CWXY}, there exists $c_2 > 0$  such that 
		\begin{equation}\label{eq:supnorm}
		\sup_{x,y\in M}\|K_t(x,y)\|\leqslant c_2\cdot \sup_{k+j \leqslant \frac{3}{2}\dim M+3}\|\widetilde{D}^{k}(\widetilde{D}e^{-t^2\widetilde{D}^2}) \widetilde{D}^{j}\|_{op}, 
		\end{equation}
		for all $x, y\in \widetilde M$, where $\|\cdot\|_{op}$ stands for the operator norm. It follows that there exist  positive numbers $c_2$,  $m_2$ and $\delta$ such that 
		\begin{equation}\label{eq:larget}
		\|K_t(x,y)\|\leqslant c_2 \frac{1}{t^{m_2}}\exp(-(\sigma_\Gamma+\delta)^2 \cdot t^2),
		\end{equation}
        for all $t>0$ and  all $x, y\in \widetilde M$.
	\end{enumerate}

In fact, since the manifolds $\widetilde M$ and $M_{\Gamma_i}$ have uniformly bounded geometry, the constants $c_{\mu, r}$, $n_1$, $m_1$, $c_2$, $m_2$ and $\delta$ from above can be chosen so that for all $i\geq 1$, we have 
\begin{equation}\label{eq:smalltN}
\|(K_i)t(x,y)\|\leqslant c_{\mu, r}\frac{(1+\dist(x,y))^{n_1}}{t^{m_1}}\exp\left( \frac{-\dist(x,y)^2}{4\mu t^2}\right),
\end{equation}
for all $t>0$ and for all $ x,y\in M_{\Gamma_i}$ with  $\dist(x,y)>r$; and 
\begin{equation}\label{eq:largetN}
\|(K_i)_t(x,y)\|\leqslant c_2 \frac{1}{t^{m_2}}\exp(-(\sigma_\Gamma+\delta)^2 \cdot t^2),
\end{equation}
 for all $t>0$ and  all $x, y\in M_{\Gamma_i}$.

For the rest of the proof, let us fix $r>0$. Note that we have 
\begin{align*}
\dist(x, \gamma y) & \leq \dist( x, \gamma x) + \dist(\gamma x, \gamma y)\\
& = \dist( x, \gamma x) + \dist(x, y)
\end{align*} 
for all $x, y\in \widetilde M $ and $\gamma \in \Gamma$.
Similarly, we have 
\[ \dist( x, \gamma x) - \dist(x, y) \leq \dist(x, \gamma y).  \]  
By line $\eqref{eq:equivwithlength}$, we have
\[ \theta_0\cdot \ell(\gamma)-c_0 - \dist(x, y)\leqslant \dist(x,\gamma y)\leqslant \theta_1\cdot \ell(\gamma)+c_1 + \dist(x, y)\] 
for all $x, y\in \widetilde M $ and $\gamma \in \Gamma$. In particular, there exist $c'_0>0$ and $c'_1>0$ such that
\begin{equation}\label{eq:length}
\theta_0\cdot \ell(\gamma)-c'_0\leqslant \dist(x,\gamma y)\leqslant \theta_1\cdot \ell(\gamma)+c'_1
\end{equation}
for all $x, y\in \mathcal F$ and $\gamma\in \Gamma$, where $\mathcal F$ is a precompact fundamental domain of $\widetilde M$ under the action of $\Gamma$. Let us define 
\[ F \coloneqq  \{ \beta\in \Gamma \mid \dist(x, \beta y) \leq r \textup{ for some } x, y \in \mathcal F\}. \]
Clearly, $F$ is a finite subset of $\Gamma$.

For any given $t>0$, it follows from line $\eqref{eq:growthofGamma}$ , $\eqref{eq:smallt}$ and $\eqref{eq:length}$ that the Schwartz kernel $K_t$ is $\ell^1$-summable (cf. Definition $\ref{def:L1}$). 

Now approximate  $f_t(x) = xe^{-t^2x^2}$ by smooth functions $\{\varphi_j\}$ whose Fourier transforms are compactly supported. By applying  the estimates in line $\eqref{eq:bound}$ and $\eqref{eq:supnorm}$ to the Schwartz kernel $K_{\varphi_j(\widetilde D)}$  (resp. $K_{\varphi_j(D_{\Gamma_i})}$) of the operator $\varphi_j(\widetilde D)$ (resp. $\varphi_j(D_{\Gamma_i})$), it is not difficult to see that $K_{\varphi_j(\widetilde D)}$  (resp. $K_{\varphi_j(D_{\Gamma_i})}$) converges to $K_t$ (resp. $(K_i)_t$) in $\ell^1$-norm (defined in  Definition $\ref{def:L1}$). Note that $\varphi_j(\widetilde D)$ has finite propagation. Since $\widetilde D$ locally coincides with $D_{\Gamma_i}$, it follows from finite propagation estimates of wave operators that  (cf.  \cite{guoxieyu}):
\[ K_{\varphi_j(D_{\Gamma_i})}(\pi_{\Gamma_i}(x), \pi_{\Gamma_i}(y)) = \sum_{\beta\in \Gamma_i}K_{\varphi_j(\widetilde D)}(x, \beta y)  \]
for all $x, y\in \widetilde M$.   
As a consequence of the above discussion, we have  
\begin{equation}\label{eq:fold}
(K_i)_{t}(\pi_{\Gamma_i}(x),\pi_{\Gamma_i}(y))=\sum_{\beta\in \Gamma_i}K_t(x,\beta y).
\end{equation}
for all $x, y\in \widetilde M$ and for all $t>0$. Furthermore, by Lemma $\ref{lemma:limitkerneltrace}$,  we have the following convergence: 
\[  \tr_{\langle\pi_{\Gamma_i}(\alpha)\rangle }(D_{\Gamma_i}e^{-t^2D^2_{\Gamma_i}}) \to  \tr_{\left\langle \alpha \right\rangle }(\widetilde De^{-t^2\widetilde D^2}), \textup{ as } j \to \infty, \]
 for each $t>0$, since $\{\Gamma_i\}$ distinguishes the conjugacy class $\langle \alpha \rangle$.

By line \eqref{eq:smallt} and \eqref{eq:larget}, there exists a positive number $c_3$ such that 
\begin{align*}
& \|K_t(x,y)\|^2 \\
&\leq c_3  \frac{(1+\dist(x,y))^{n_1}}{t^{m_1+m_2}}\exp\left( \frac{-\dist(x,y)^2}{4\mu t^2}\right) e^{-(\sigma_\Gamma+\varepsilon)^2 \cdot t^2}  e^{-\varepsilon^2 t^2}\\
& \leq c_3 \frac{(1+\dist(x,y))^{n_1}}{t^{m_1+m_2}}\exp\left( \frac{-\dist(x,y)\cdot (\sigma_\Gamma+\varepsilon)}{\mu}\right) e^{-\varepsilon^2 t^2},
\end{align*}
 for all $x, y\in \widetilde M$ with $\dist(x, y)>r$, where $\varepsilon = \delta/2$.
By choosing $\mu >1$ sufficiently close to $1$, we see that there exist  $c_4 >0 $ and $\lambda >1$ such that 
\[  \|K_t(x,y)\|\leq c_4 \frac{e^{-\varepsilon^2 t^2}}{t^{(m_1+m_2)/2}} \exp(- \frac{\lambda}{2} \cdot \sigma_\Gamma\cdot \dist(x, y))  \]
 for all $x, y\in \widetilde M$ with $\dist(x, y)>r$.  It follows that there exist $c_5>0$ and $m>0$ such that 
\begin{align*}
& \sum_{\gamma\in \Gamma}\|K_t(x,\gamma y)\|\\
&\leqslant  \sum_{\gamma\in F} c_2 \frac{e^{-(\sigma_\Gamma+\delta)^2 \cdot t^2}}{t^{m_2}}
+  c_4 \frac{e^{-\varepsilon^2 t^2}}{t^{(m_1+m_2)/2}}\sum_{\gamma\notin F}e^{-\lambda \cdot  K_{\Gamma}\cdot (\ell(\gamma)- c'_0 \tau_0^{-1})} \\
& \leqslant   c_2 \cdot |F|\cdot \frac{e^{-\varepsilon^2 t^2}}{t^{m_2}}
+  c_4 \frac{e^{-\varepsilon^2 t^2}}{t^{(m_1+m_2)/2}}\sum_{n=0}^\infty e^{K_{\Gamma}\cdot n}e^{-\lambda \cdot  K_{\Gamma}\cdot (n-c'_0\tau_0^{-1})}\\
&<\frac{c_5}{t^{m }}e^{-\varepsilon^2 t^2}
\end{align*}
for all $t>0$. In particular, we have 
\[ \sum_{\gamma\in \langle \alpha \rangle } \big|\tr K_t(x,\gamma y) \big| < \frac{c_5}{t^{m }}e^{-\varepsilon^2 t^2}  \]
for all $x, y\in \widetilde M$ and all $t>0$. By the same argument, we also have 
\[ \sum_{\omega\in \langle \pi_{\Gamma_i}(\alpha) \rangle } \big|\tr (K_i)_t(x',\omega y') \big| < \frac{c_5}{t^{m }}e^{-\varepsilon^2 t^2}  \]
for all $x', y'\in M_{\Gamma_i}$ and all $t>0$. Therefore, the functions  
\[ |\tr_{\langle\alpha \rangle }(\widetilde D e^{-t^2\widetilde D^2})| \textup{ and } |\tr_{\langle\pi_{\Gamma_i}(\alpha )\rangle }(D_{\Gamma_i}e^{-t^2D^2_{\Gamma_i}})| \]
are all bounded by the function 
\[ c_5\cdot t^{-m}e^{-\varepsilon^2 t^2}.\]
The latter is clearly absolutely integrable on $[1, \infty)$. 

We have found above an appropriate dominating function on the interval $[1, \infty)$.  Now let us find the dominating function on  $(0, 1]$. Since $\Gamma$ acts on $\widetilde M$ freely and cocompactly, it follows that  there exists $\varepsilon_0 >0$ such that 
\[ \dist(x, \gamma x) > \varepsilon_0 \]
for all $x\in \widetilde M$ and all $\gamma \neq e \in \Gamma$.
By applying line $\eqref{eq:smallt}$, a similar calculation as above shows that there exist $\varepsilon_1>0$ and $c_6 >0$ such that 
\[  \sum_{\gamma\in \Gamma}\|K_t(x,\gamma x)\| < \frac{c_6}{t^{m_1}} e^{-\varepsilon_1\cdot  t^{-2}} \]
for all $x\in \mathcal F$ and all $t\leq 1$. The same estimate also holds for $(K_i)_t$.  Therefore, on the interval $(0, 1]$, the functions  
\[ |\tr_{\langle\alpha \rangle }(\widetilde D e^{-t^2\widetilde D^2})| \textup{ and } |\tr_{\langle\pi_{\Gamma_i}(\alpha )\rangle }(D_{\Gamma_i}e^{-t^2D^2_{\Gamma_i}})| \]
are all bounded by the function  \[ c_6 \cdot t^{-m_1} e^{-\varepsilon_1\cdot  t^{-2}}.\] The latter is absolutely integrable on $(0, 1]$. This finishes the proof.

\end{proof}

If the group $\Gamma$ has subexponential growth, then it follows from Definition $\ref{def:cst}$ that $\sigma_\Gamma = 0$. In this case, if $\widetilde D$ has a spectral gap, then it is automatically sufficiently large, hence the following immediate corollary. 

\begin{corollary}\label{cor:sub}
	With the above notation, suppose $\{\Gamma_i\}$ distinguishes the conjugacy class $\langle \alpha \rangle$ of a non-identity element $\alpha\in \Gamma$. If $\Gamma$ has subexponential growth and $\widetilde D$ has a spectral gap at zero , then we have
	\[ \lim_{i\to \infty}\eta_{\langle\pi_{\Gamma_i}(\alpha)\rangle}(D_{\Gamma_i}) = \eta_{\left\langle \alpha \right\rangle }(\widetilde D). \]
\end{corollary}

\begin{remark}
 Recall that if $\Gamma$ has subexponential growth, then $\Gamma$ is amenable and thus a-T-menable \cite{MR1388307}. So the convergence of   the limit
	\[  \lim_{i\to\infty}\eta_{\langle \pi_{\Gamma_i}(\alpha)\rangle }(D_{\Gamma_i}) \]
also follows from Proposition $\ref{prop:atmenable}$. In fact, Proposition $\ref{prop:atmenable}$ implies that the above limit stabilizes.  On the other hand,  Corollary $\ref{cor:sub}$ above answers positively both part (I) and part (II) of Question $\ref{question}$.
\end{remark}

\begin{remark}
In this remark, we shall briefly comment on the condition that the spectral gap of $\widetilde D$ at zero is greater than $\sigma_\Gamma$ in Theorem $\ref{main}$. Here is a class of natural examples such that the spectral gap of $\widetilde D$ at zero is greater than $\sigma_\Gamma$ and the higher rho invariant\footnote{As we have seen in the proofs of Theorem $\ref{prop:atmenable}$ and Theorem $\ref{thm:athyper}$, the delocalized eta invariant $\eta_{\langle \alpha \rangle}(\widetilde 
	D)$ is essentially the pairing between the higher rho invariant $\rho(\widetilde D)$  and the delocalized trace $\tr_{\langle \alpha \rangle}$, cf. \cite[Theorem 4.3]{Xie}. } $\rho(\widetilde D)$ is nonzero.

\added{ Suppose that $N$ is a closed spin manifold  equipped with a positive scalar curvature metric $g_N$, whose fundamental group $F = \pi_1(N)$ is finite and its higher rho invariant $\rho(\widetilde D_N)$ is nontrivial. Here $\widetilde D_N$  is the Dirac operator on the universal covering $\widetilde N$ of $N$. For instance, let  $N$ to be a lens space, that is, the quotient of the 3-dimensional sphere by a free action of a finite cyclic group. In this case, the classical equivariant Atiyah-Patodi-Singer index theorem implies that the delocalized higher rho invariant of $N$ is nontrivial, cf. \cite{MR511246}. }  
	
\added{	Now let $X$ be an even dimensional closed spin manifold, whose Dirac operator $D_X$ has nontrivial higher index in $K_0(C^\ast_r(\Gamma))$, where $\Gamma = \pi_1(X)$. In particular, it follows that $D_X$ defines a nonzero element in  the equivariant \mbox{$K$-homology} $K_0(C^\ast_L(E\Gamma)^\Gamma)$ of the universal space $E\Gamma$ for free $\Gamma$ actions.  Consider the product  space $M = V\times N$  equipped with a metric $g_{M}=g_X+\varepsilon\cdot  g_N$, where $g_X$ is an arbitrary Riemannian metric on $X$ and the metric $g_N$ on $N$ is scaled by a positive number $\varepsilon$. Denote the Dirac operator on the universal covering $\widetilde M$ of $M$ by $\widetilde D_M$. The spectral gap of $\widetilde D_M$ at zero can always be made sufficiently large, as long as we choose $\varepsilon$ to be  sufficiently small. To see that  $\rho(\widetilde D_M)$ is nonzero in  $K_1(C_{L,0}^\ast(E(\Gamma\times F))^{\Gamma\times F})$, we apply the product formula for secondary invariants (cf. \cite[Claim 2.19]{Xiepos}, \cite[Corollary 4.15]{MR3551834}), which states  that the higher rho invariant $\rho(\widetilde D_M)$ is the product of the $K$-homology class of $D_X$ and the higher rho invariant $\rho(\widetilde D_N)$. It follows from the above construction that the higher rho invariant  $\rho(\widetilde D_M)$ is nonzero in  $K_1(C_{L,0}^\ast(E(\Gamma\times F))^{\Gamma\times F})$.  In fact, if the Baum-Connes conjecture holds for $\Gamma$, then the  $K$-theory group $K_1(C_{L,0}^\ast(E(\Gamma\times F))^{\Gamma\times F})$ is (at least rationally) generated by the higher rho invariants of the above examples, cf. \cite[Theorem 3.7 \& Corollary 3.16]{Xie:2017aa}.  }

\added{On the other hand, we also would like to point out that for the operator $\widetilde D_M$ from the above examples, a straightfoward calculation shows the delocalized eta invariant $\eta_{\langle \alpha \rangle}(\widetilde D_M)$ could be  nonzero only in the case  when $\alpha = (1, a) \in \Gamma \times F$ with $a$ being a non-identity element of $F = \pi_1(N)$. This essentially reduces the computation of $\eta_{\langle \alpha \rangle}(\widetilde D_M)$ to the case of finite fundamental groups. Consequently, Question $\ref{question}$ has a positive answer for the above examples but for trivial reasons.  
}
\end{remark}

By the proof of Theorem $\ref{main}$ above, in order to bound the function 
\[ |\tr_{\langle\alpha \rangle }(\widetilde D e^{-t^2\widetilde D^2})|,  \]
it suffices to assume the spectral gap of $\widetilde D$ at zero to be greater than 
\[ \sigma_{\langle \alpha\rangle} \coloneqq \frac{2\cdot K_{\langle \alpha \rangle}}{\theta_0}, \]
where $\theta_0$ is the constant from line $\eqref{eq:distortion}$ and   $K_{\langle \alpha \rangle}$ is nonnegative constant such that there exists some constant $C>0$ satisfying
\begin{equation}
\#\{\gamma\in \langle \alpha\rangle :\ \ell(\gamma)\leqslant n\}\leqslant Ce^{K_{\langle \alpha \rangle }\cdot n}
\end{equation} 
for all $n$. In fact, if we have a uniform control of the spectral gap of $D_{\Gamma_i}$ at zero and the growth rate of the conjugacy class $\{\langle \pi_{\Gamma_i}(\alpha) \rangle\}$ for all $i\geq 1$, a notion to be made precise in the following,  then the same proof above also implies that 
\[ 	\lim_{i\to\infty}\eta_{\langle \pi_{\Gamma_i}(\alpha)\rangle }(D_{\Gamma_i})=\eta_{\left\langle \alpha\right\rangle }(\widetilde D)\]
in this case. 

Recall that $S$ a symmetric finite generating set of $\Gamma$. For each  normal subgroup $\Gamma_i$ of $\Gamma$,  the map $\pi_{\Gamma_i}\colon \Gamma\to \Gamma/\Gamma_i$ is the canonical quotient map. The set  $\pi_{\Gamma_i}(S)$ is a symmetric generating set for $\Gamma/\Gamma_i$, hence induces a word length function $\ell_{\Gamma_i}$ on $\Gamma/{\Gamma_i}$. More explicitly, we have  
\begin{equation}
\ell_{\Gamma_i}(\omega):=\inf\{\ell(\beta) : \beta \in\pi_{\Gamma_i}^{-1}(\omega)\}.
\end{equation}
for all $\omega\in \Gamma/\Gamma_i$.

\begin{definition}
For a given conjugacy class $\langle \alpha \rangle$ of $\Gamma$, we say that $\langle \alpha \rangle$ has uniform exponential growth with respect to a family of normal subgroups $\{ \Gamma_i\}$, if there exist $C>0$ and $A\geq 0$ such that 
	\begin{equation}
 \#\big\{\omega \in \langle \pi_{\Gamma_i}(\alpha )\rangle:\ell_{\Gamma_i}(\omega)\leqslant n\big\}\leqslant C e^{A\cdot n}.
	\end{equation}
	for all $i\geq 1$ and all $n\geq 0$. 
	In this case, we define $K_u$ to be the infimum of all such numbers $A$. 
\end{definition}

\begin{definition}
	With the above notation, we define
	\[\sigma_u \coloneqq \frac{2\cdot K_u}{\theta_0}, \]
	where $\theta_0$ is the constant from line $\eqref{eq:distortion}$. 
\end{definition}
The same argument from the proof of Theorem $\ref{main}$ can be used to prove the following. 
\begin{theorem}\label{thm:conjcontrol}
With the same notation as in Theorem $\ref{main}$, suppose $\{\Gamma_i\}$ is a sequence of finite-index normal subgroups that distinguishes the conjugacy class $\langle \alpha\rangle$ of a non-identity element $\alpha\in \Gamma$. Assume $\langle\alpha\rangle$ has uniform exponential growth with respect to $\{\Gamma_i\}$.  If  there exists $\varepsilon >0$ such that the spectral gap of $D_{\Gamma_i}$ at zero is greater than $\sigma_u +\varepsilon$ for sufficiently large $i\gg 1$, then 
		\begin{equation*}
		\lim_{i\to\infty}\eta_{\langle \pi_{\Gamma_i}(\alpha)\rangle }(D_{\Gamma_i})=\eta_{\left\langle \alpha \right\rangle }(\widetilde D).
		\end{equation*}
\end{theorem}

\begin{remark}\label{rm:gap}
In fact, if the spectral gap of $\widetilde D$ at zero is greater than $\sigma_u +\varepsilon$ for some positive number $\varepsilon$, then the spectral gap of $D_{\Gamma_i}$ at zero to be greater than $\sigma_u +\varepsilon$ for all $i\geq 1$. \added{Indeed, it was shown  in \cite{guoxieyu} that  the spectral gap of $\widetilde D$ at zero is the same as the spectral gap of $\widetilde  D_{\max}$, where $\widetilde  D_{\max}$ is the operator $\widetilde D$ viewed as unbounded self-adjoint multiplier of the maximal equivariant Roe algebra $C^\ast_{\max}(\widetilde M)^\Gamma \cong \mathcal K\otimes C_{\max}^\ast(\Gamma)$. It follows that if there exists a postive number $\varepsilon$ such that the spectral gap of $\widetilde D$ at zero is greater than $\sigma_u + \varepsilon$, then the spectral gap of $D_{\Gamma_i}$ at zero is greater than $\sigma_u +\varepsilon$ for all $i\geq 1$.  }

\end{remark}

\section{Separation rates of conjugacy classes}\label{sec:sep}
In this section, we introduce a notion of separation rate for how fast a sequence of normal subgroups $\{\Gamma_i\}$ of $\Gamma$ distinguishes a conjugacy class $\langle \alpha \rangle $ of $\Gamma$ and use it to answer Question $\ref{question}$ in some cases. 
\begin{definition}\label{def:fastsep}
	For  each normal subgroup $\Gamma'$ of $\Gamma$, let  $\pi_{\Gamma'} \colon \Gamma \to \Gamma/\Gamma'$ be the quotient map from $\Gamma$ to $\Gamma/\Gamma'$. Given a conjugacy class $\langle \alpha \rangle$ of $\Gamma$,   we define the injective radius of  $\pi_{\Gamma'}$ with respect to $\langle \alpha\rangle $  to be 
	\begin{equation}\label{eq:s(N)}
	r(\Gamma')\coloneqq \max\{n \mid \textup{ if } \gamma \notin\langle \alpha \rangle \textup{ and } \ell(\gamma) \leq n, \textup{ then }
	\pi_{\Gamma'}(\gamma )\notin\langle\pi_{\Gamma'}(\alpha)\rangle\}.
	\end{equation}
\end{definition}
\begin{definition}
	Suppose that $\{\Gamma_i\}$ is a sequence of finite-index normal subgroups of $\Gamma$ that distinguishes $\langle \alpha \rangle$. We say that $\{\Gamma_i\}$ distinguishes $\langle \alpha \rangle$ sufficiently fast if there exist $C>0$ and $R>0$ such that
	\begin{equation}\label{eq:seqspeed}
	|\langle\pi_{\Gamma_i}(\alpha)\rangle|\leqslant Ce^{R\cdot r({\Gamma_i})}.
	\end{equation}
	In this case, we define the separation rate $R_{\langle \alpha \rangle, \{\Gamma_i\}}$ of $\langle \alpha\rangle$ with respect to  $\{\Gamma_i\}$ to be the infimum of all such numbers $R$. 
\end{definition}
We have the following proposition. 
\begin{proposition}\label{prop:sep}
	Let $\langle \alpha\rangle$ be the conjugacy class of a non-identity element $\alpha\in \Gamma$.  Suppose $\{\Gamma_i\}$ is a sequence of finite-index normal subgroups that distinguishes $\langle \alpha \rangle$  sufficiently fast with separation rate $R =  R_{\langle \alpha \rangle, \{\Gamma_i\}}$. If $\eta_{\langle \alpha \rangle}(\widetilde D)$ is finite\footnote{To be precise, $\eta_{\langle \alpha \rangle}(\widetilde D)$ is finite if the integral in line $\eqref{delocalizedeta}$ converges. In particular, the integral in line $\eqref{delocalizedeta}$  does \emph{not}  necessarily absolutely converge. }  and there exists  $\varepsilon >0$ such that the spectral gap of $D_{\Gamma_i}$ at zero is greater than\footnote{For example, if the spectral gap of $\widetilde D$ at zero is greater than $\sigma_R +\varepsilon$ for some positive number $\varepsilon$, then the spectral gap of $D_{\Gamma_i}$ at zero to be greater than $\sigma_R +\varepsilon$ for all $i\geq 1$, cf. Remark \ref{rm:gap}.} $\sigma_R + \varepsilon$ for all sufficiently large $i\gg 1$, where \[ \sigma_R \coloneqq  \frac{2 (K_\Gamma \cdot R)^{1/2}}{\theta_0},  \] then we have
	$$\lim_{i\to\infty}\eta_{\langle \pi_{\Gamma_i}(\alpha)\rangle }(D_{\Gamma_i})=\eta_{\left\langle \alpha \right\rangle }(\widetilde D).$$
\end{proposition}
\begin{proof}
	By assumption, the integral 
	\[ \eta_{\left\langle \alpha \right\rangle }(\widetilde D)\coloneqq 
	\frac{2}{\sqrt\pi}\int_{0}^\infty \tr_{\left\langle \alpha \right\rangle }(\widetilde De^{-t^2\widetilde D^2})dt\]
	converges. To prove the proposition, it suffices to show that there exists a sequence of positive real numbers $\{s_i\}$ such that  $s_i\to\infty$, as $i\to\infty$, and 
	\begin{enumerate}
		\item \begin{equation}\label{eq:smalllim}
		\lim_{i\to\infty} \int_0^{s_i}\left(\tr_{\langle \pi_{\Gamma_i}(\alpha)\rangle}(D_{\Gamma_i}e^{-t^2D^2_{\Gamma_i}})-\tr_{\langle \alpha\rangle}(\widetilde De^{-t^2\widetilde D^2}) \right)dt=0
		\end{equation}
		
		\item and \begin{equation}\label{eq:largelim}
		\lim_{i\to\infty}\int_{s_i}^\infty\tr_{\langle \pi_{\Gamma_i}(\alpha)\rangle}(D_{\Gamma_i}e^{-t^2D^2_{\Gamma_i}})dt=0.
		\end{equation}
	\end{enumerate}

	Let $K_t(x,y)$ (resp. $(K_i)_t(x,y)$) be the Schwartz kernel of $\widetilde De^{-t^2\widetilde D^2}$ (resp. $D_{\Gamma_i} e^{-t^2D_{\Gamma_i}^2}$). Recall that we have (cf. line $\eqref{eq:fold}$)  
	\[ (K_i)_t(\pi_{\Gamma_i}(x),\pi_{\Gamma_i}(y))=\sum_{\beta\in \Gamma_i}K_t(x,\beta y) \]
	for all $x, y\in \widetilde M$. 
It follows that 
	\begin{align*}
	&\big|\tr_{\langle \pi_{\Gamma_i}(\alpha )\rangle}(D_{\Gamma_i}e^{-t^2D^2_{\Gamma_i}})-\tr_{\langle \alpha\rangle}(\widetilde De^{-t^2\widetilde D^2})\big|\\
	\leqslant &\sum_{\substack{\gamma\in \pi_{\Gamma_i}^{-1}\langle\pi_{\Gamma_i}(\alpha )\rangle \\ 
			\textup{but } \gamma \notin \langle \alpha \rangle}}\int_{x\in\mathcal F}\|K_t(x, \gamma x)\|dx.
	\end{align*}
	By the definition of $r(\Gamma_i)$ in line \eqref{eq:s(N)}, we see that
	$$ \{ \gamma \in \Gamma \mid \gamma \in \pi_{\Gamma_i}^{-1}\langle\pi_{\Gamma_i}(\alpha)\rangle \textup{ but }\gamma\notin  \langle \alpha \rangle \}\subseteq\{\gamma\in \Gamma \mid \ell(\gamma)\geqslant r(\Gamma_i)\}.$$	
 From line \eqref{eq:smallt}, we have
	\begin{align*}
	\sum_{\ell(\gamma)\geqslant r(\Gamma_i)}\int_{x\in\mathcal F}\|K_t(x,\gamma x)\|dx
	\leqslant c_{\mu, r}\sum_{m=r(\Gamma_i)}^\infty e^{-\frac{(\theta_0\cdot  m-c_0)^2}{4\mu t^2}}e^{K_\Gamma \cdot m}.
	\end{align*}
Note that 
		\begin{align*}
	&\int_0^{s_i}\sum_{m=r(\Gamma_i)}^\infty e^{-\frac{(\theta_0\cdot m-c_0)^2}{4\mu t^2}}e^{K_\Gamma \cdot m}dt
	\leqslant s_i\sum_{m=r(\Gamma_i)}^\infty e^{-\frac{(\theta_0\cdot m-c_0)^2}{4\mu s_i^2}+K_\Gamma \cdot m}.
	\end{align*}
The right hand side goes to zero, as $s_i\to\infty$, as long as there exists $\lambda_1 >1$ such that
\begin{equation}\label{eq:upper}
\frac{(\theta_0\cdot r(\Gamma_i) - c_0)^2}{4\mu s_i^2} > \lambda_1\cdot  K_\Gamma \cdot r(\Gamma_i)
\end{equation}
for all sufficiently large $i\gg 1$.  Since $\{\Gamma_i\}$ distinguishes $\langle \alpha \rangle$, we have that $r(\Gamma_i)\to \infty $, as $i\to\infty$. So the condition in line $\eqref{eq:upper}$ is equivalent to 
\[ s_i^2 < \frac{\theta_0^2 \cdot r(\Gamma_i)}{4\mu \cdot \lambda_1 \cdot K_\Gamma} \]
for sufficiently large $i\gg 1$. 

On the other hand, by the inequality from line \eqref{eq:largetN}, there exist $c>0$ and $\varepsilon>0$ such that  
	\[ \left|\int_{s_i}^\infty\tr_{\langle \pi_{\Gamma_i}(\alpha)\rangle}(D_{\Gamma_i}e^{-t^2D^2_{\Gamma_i}})dt\right|\leqslant
	c\cdot e^{-(\sigma_R+\varepsilon)^2 \cdot s_i^2} \cdot |\langle\pi_{\Gamma_i}(\alpha)\rangle|\]
	for all sufficiently large $i\gg 1$. 
Note that the right hand side goes to zero, as $s_i \to \infty$, as long as there exists $\lambda_2 >1$ such that
\begin{equation}\label{eq:lower}
(\sigma_R+\varepsilon)^2\cdot s_i^2 > \lambda_2\cdot  R \cdot r(\Gamma_i).
\end{equation}
	for all sufficiently large $j\gg 1$.
Combining the two inequalities in line $\eqref{eq:upper}$ and $\eqref{eq:lower}$ together, we can choose a sequence of real numbers $\{s_i\}$ that satisfies  the limits in both line $\eqref{eq:smalllim}$ and $\eqref{eq:largelim}$,  as long as there exists  $\lambda_3 >1$ such that 
\begin{equation}\label{eq:specgap}
\frac{\theta_0^2 \cdot r(\Gamma_i)}{4\mu \cdot \lambda_1 \cdot K_\Gamma} > \lambda_3 \frac{R\cdot r(\Gamma_i) }{(\sigma_R+\varepsilon)^2} 
\end{equation}
for all sufficiently large $i\gg 1$.  By choosing $\mu$ sufficiently close to $1$, the inequality in line $\eqref{eq:specgap}$ follows from the definition of $\sigma_R$. This finishes the proof.  
\end{proof}

We finish this section with the following calculation of the separation rates of conjugacy classes of $\SL_{2 }(\mathbb Z)$. 
	The group $\SL_2(\mathbb Z)$ $\SL_2(\mathbb Z)$ is  a conjugacy separable group \cite{Stebe}. It has a presentation: 
	$$\langle x,y\ |\ x^4=1,\ x^2=y^3\rangle,$$
	where $x=\begin{psmallmatrix}
	0&-1\\1&0
	\end{psmallmatrix}$ and $y=\begin{psmallmatrix}
	0&-1\\1&1
	\end{psmallmatrix}$.
In particular, it follows that $\SL_2(\mathbb Z)$ is an amalgamated free product of a cyclic group of order $4$ and a cyclic group of order $6$. We now show that  for any finite order element $\alpha\in \SL_2( \mathbb Z)$, there exists a sequence of finite-index normal subgroups $\{\Gamma_i\}$ of $\SL_2(\mathbb Z)$ that distinguishes $\langle \alpha\rangle$ such that the corresponding separation rate  $R_{\langle \alpha \rangle, \{\Gamma_i\}} = 0$.
	
Since $\SL_2(\mathbb Z)$ is an amalgamated free product of a cyclic group of order $4$ and a cyclic group of order $6$, every finite order element in $\SL_2(\mathbb Z)$ is  conjugate to an element in one of the factors. It follows that every finite order element of  $\SL_2(\mathbb Z)$  is conjugate to a power of $x$ and $y$.   	Let $\psi\colon \SL_2(\mathbb Z)\to\mathbb Z/12\mathbb Z$ be the group homomorphism defined by  $\psi(x)=3$ and $\psi(y)=2$. In particular, we have 
\[ \psi(e) = 0, \psi(x) = 3, \psi(x^2) = \psi(y^3) = 6, \psi(x^3)= 9,\]
\[ \psi(y) = 2, \psi(y^2) = 4, \psi(y^4)= 8, \textup{ and } \psi(y^5) = 10.  \]  
It follows that any finite order elements $\gamma_1$ and $\gamma_2$ of $\SL_2( \mathbb Z)$ are conjugate in $\SL_2( \mathbb Z)$ if and only if $\psi(\gamma_1) = \psi(\gamma_2)$.

%

	 Now given any finite-index normal subgroup $N$ of $\SL_2(\mathbb Z)$, the group $ N_1 = N\cap \ker(\psi)$ is a finite-index normal subgroup of $\SL_2(\mathbb Z)$.  By the discussion above, we see that any finite order elements $\gamma_1$ and $\gamma_2$ of $\SL_2(\mathbb Z)$ are conjugate in $\SL_2(\mathbb Z)$ if and only if they are conjugate in $\SL_2(\mathbb Z)/N_1$. In other words, the set $\{N_1\}$ consisting of a single finite-index normal subgroup distinguishes the conjugacy class $\langle \alpha \rangle$ of any finite order element $\alpha\in \SL_2(\mathbb Z)$. Moreover, the injective radius $r(N_1)$ of $\pi_{N_1}\colon \SL_2(\mathbb Z) \to \SL_2(\mathbb Z)/N_1$ with respect to  $\langle \alpha\rangle $ is infinity. It follows that the separation rate $R_{\langle \alpha \rangle, \{N_1\}} = 0$ in this case.

\begin{remark}
	Since $\SL_2(\mathbb Z)$ is hyperbolic, Puschnigg's smooth dense subalgebra $\mathcal A$ of $C_r^\ast(\SL_2(\mathbb Z))$ admits a continuous extension of the trace map $\tr_{\langle \alpha \rangle}$ for any conjugacy class $\langle \alpha \rangle$ of $\SL_2(\mathbb Z)$ (cf. \cite{Puschnigg}). In this case, for any element $\alpha\neq e \in \Gamma$, the delocalized eta invariant  $\eta_{\langle \alpha \rangle}(\widetilde D)$ is finite\footnote{For any hyperbolic group and the conjugacy class of any nonidentity element, the integral in line $\eqref{delocalizedeta}$ absolutely converges.} (cf. \cite[Section 4]{Lott}\cite[Section 6]{CWXY}). Hence we can apply Proposition $\ref{prop:sep}$ to answer positively  both part (I) and (II) of Question $\ref{question}$ for the group $\SL_2(\mathbb Z)$, when $\alpha$ is a finite order element. 
	
	On other hand, since  $\SL_2(\mathbb Z)$ is also a-T-menable, we can equally apply Proposition $\ref{prop:atmenable}$ to answer positively  both part (I) and (II) of Question $\ref{question}$ for the group $\SL_2(\mathbb Z)$ (cf. the discussion at the end of Section $\ref{sec:max}$). 
\end{remark}
\appendix

\section{Positive scalar curvature and the Stolz conjecture}\label{sec:stolz}

In this appendix, we shall explain how the geometric conditions given in Definition \ref{def:bound}  \& \ref{def:properbound}  are related to the reduced Baum-Connes conjecture and the Stolz conjecture on positive scalar curvature metrics.

\begin{definition}
	Given a topological space $Y$, let	$R^\spin_n(Y)$ be the following bordism group of triples $(L, f, h)$, where $L$ is an $n$-dimensional compact spin manifold (possibly with boundary), $f: L\to Y$ is a continuous map, and $h$ is a positive scalar curvature metric on the boundary $\partial M$. Two triples $(L_1, f_1, h_1)$ and $(L_2, f_2, h_2)$ are bordant if 
	\begin{enumerate}[label=(\alph*), ref=(\alph*)]
		\item there is a bordism $(V, F, H)$ between $(\partial L_1, f_1, h_1)$ and $(\partial L_2, f_2, h_2)$ such that $H$ is  a positive scalar curvature metric on $V$ with product structure near $\partial L_i$ and  $H|_{\partial L_i} = h_i$, and the restriction of the map $F\colon V \to Y$ on $\partial L_i$ is $f_i$;
		\item and the closed spin manifold $L_1\cup_{\partial L_1} V \cup_{\partial L_2} L_2$ (obtained by gluing $L_1, V$ and $L_2$ along their common boundaries) is the boundary of a spin manifold $W$ with a map $E: W\to Y$ such that $E|_{ L_i} = f_i$ and $E|_{V} = F$. 
	\end{enumerate} 
\end{definition}

The above definition has the following obvious analogue for the case of proper actions. 
\begin{definition}
	Let $X$ be a proper metric space equipped with a proper and cocompact isometric action of a discrete group $\Gamma$.  We denote by $R^{\spin}_n(X)^\Gamma$ the set of bordism classes of pairs $(L, f, h)$, where $L$ is an $n$-dimensional complete spin manifold  equipped  with a proper and cocompact isometric action of $\Gamma$, the map $f: L\to X$ is a $\Gamma$-equivariant  continuous map and $h$ is a $\Gamma$-invariant positive scalar curvature metric on $\partial L$. Here the bordism equivalence relation is defined similarly as the non-equivariant case above. 
\end{definition}

If the action of $\Gamma$ on $X$ is free and proper, then it follows by definition that 
\[  R_n^\spin(X)^\Gamma \cong  R_n^\spin(X/\Gamma).  \]

Suppose $(L, f, h)$ is an element in $R_n^\spin(E\Gamma)^\Gamma \cong R_n^\spin(B\Gamma)$, where $B\Gamma = E\Gamma/\Gamma$ is the classifying space for free $\Gamma$-actions. Let $ L_\Gamma$ be the $\Gamma$-covering space of $M$ induced by the map $f\colon L\to B\Gamma$ and $D_{L_\Gamma}$ be the associated Dirac operator.  Due to the positive scalar curvature metric $h$ on $\partial L$, the $\Gamma$-equivariant operator $D_{L_\Gamma}$ has a well-defined higher index class $\ind(D_{L_\Gamma})$ in $K_n(C^\ast_r(\Gamma))$, cf. \cite[Proposition 3.11]{Roe} \cite{MR3439130}. By the relative higher index theorem \cite{UB95,  MR3122162}, we have the following well-defined index map 
\[ \ind\colon  R^{\spin}_n(B\Gamma) \to KO_n(C^\ast_r(\Gamma; \mathbb R)), \quad (L, f, h) \mapsto  \ind(D_{L_\Gamma}),\] 
where $C_r^\ast(\Gamma; \mathbb R)$ is the reduced group $C^\ast$-algebra of $\Gamma$ with real coefficients. 
	Now let $\mathfrak B$ be the Bott manifold, a simply connected spin manifold of dimension $8$ with $\widehat A(\mathfrak B) =1$. This manifold is not unique, but any choice will work for the following discussion. To make the discussion below more transparent, let us choose a $\mathfrak B$ that is equipped with a scalar flat curvature metric. The fact that such a choice exists follows for example from the work of Joyce \cite[Section 6]{MR1383960}.

Let  $(L, f, h)$ be an element $ R^{\spin}_n(B\Gamma)$, that is, $L$ is an $n$-dimensional spin manifold whose boundary $\partial L$ carries a positive scalar curvature metric $h$, together with a map $f\colon L \to B\Gamma$.   Taking direct product with $k$ copies of $\mathfrak B$ produces an element  $(L', f', h')$ in $ R^{\spin}_{n+8k}(B\Gamma)$, where $L'  = L\times \mathfrak B\times \cdots \times \mathfrak B$, $f' = f\circ p$ with the map $p$ being the projection from $L'$ to $ L$, and $h'$ is the product metric of $h$ with the Riemannian metric on $\mathfrak B$. By our choice of $\mathfrak B$ above, the Riemannian metric $h'$ also has positive scalar curvature since $h$ does.  Define $R^{\spin}_n(B\Gamma)[\mathfrak B^{-1}] $ to be  the direct limit of the following directed system: 
\[  R^{\spin}_n(B\Gamma)\xrightarrow{\times \mathfrak B} R^{\spin}_{n+8}(B\Gamma)\xrightarrow{\times \mathfrak B} R^{\spin}_{n+16}(B\Gamma) \to \cdots. \] Since the higher index class  $\ind(D_{L_\Gamma})$ associated to $(L, f, h)$ is invariant under taking direct product with $\mathfrak B$,  it follows that the above index map induces the following well-defined index map: 
\[ \theta\colon  R^{\spin}_n(B\Gamma)[\mathfrak B^{-1}] \to KO_n(C^\ast_r(\Gamma; \mathbb R)), \quad (L, f, h) \mapsto  \ind(D_{L_\Gamma}).\]

\begin{conjecture}[{Stolz conjecture \cite{SS95,stolz-concordance}}]\label{conj:stolz}
	The index map \[ \theta\colon  R^{\spin}_n(B\Gamma)[\mathfrak B^{-1}] \to KO_n(C^\ast_r(\Gamma; \mathbb R))\]
	is an isomorphism.  
\end{conjecture}

Similarly, if one works with the universal space $\underline{E}\Gamma$  for proper $\Gamma$-actions instead, then the same argument from above also produces a similar index map 
\[ \Theta\colon  R^{\spin}_n(\underline{E}\Gamma)^\Gamma[\mathfrak B^{-1}] \to KO_n(C^\ast_r(\Gamma; \mathbb R))\] 
where $R^{\spin}_n(\underline{E}\Gamma)^\Gamma[\mathfrak B^{-1}] $ is the direct limit of the following directed system: 
\[  R^{\spin}_n(\underline{E}\Gamma)^\Gamma\xrightarrow{\times \mathfrak B} R^{\spin}_{n+8}(\underline{E}\Gamma)^\Gamma\xrightarrow{\times \mathfrak B} R^{\spin}_{n+16}(\underline{E}\Gamma)^\Gamma \to \cdots. \]
One has the following analogue of the Stolz conjecture above, which will be called the generalized Stolz conjecture from now on. 

\begin{conjecture}[Generalized Stolz conjecture]\label{conj:stolz2}
	The index map \[ \Theta\colon  R^{\spin}_n(\underline{E}\Gamma)^\Gamma[\mathfrak B^{-1}] \to KO_n(C^\ast_r(\Gamma; \mathbb R))\] 
	is an isomorphism.  
\end{conjecture}

\added{If $\Gamma$ is torsion-free, then clearly the generalized Stolz conjecture coincides with the original Stolz conjecture. By definition, the surjectivity of the Stolz map $\theta$ in Conjecture $\ref{conj:stolz}$ implies the surjectivity the generalized Stolz map $\Theta$ in Conjecture $\ref{conj:stolz2}$. On the other hand,  the surjectivity of the Stolz map $\theta$ follows from the surjectivity of the Baum-Connes assembly map 
	\[ \mu_{\mathbb R}\colon   KO_\bullet^\Gamma(\underline{E}\Gamma)\to KO_\bullet(C_{r}^*(\Gamma;\mathbb R)) \] 
cf.  \cite[Corollary 3.15]{Xie:2017aa}. At the time of writing of this paper, the injectivity of the Stolz map $\theta$ or generalized Stolz map $\Theta$ is wild open, and it is not even known in the case where $\Gamma$ is the trivial group. }

\added{Similarly, one could also formulate the maximal version of the (generalized) Stolz conjecture by considering the index maps  \[ \theta_{\max}\colon  R^{\spin}_n(B\Gamma)[\mathfrak B^{-1}] \to KO_n(C^\ast_{\max}(\Gamma; \mathbb R))\]  and \[ \Theta_{\max}\colon  R^{\spin}_n(\underline{E}\Gamma)^\Gamma[\mathfrak B^{-1}] \to KO_n(C^\ast_{\max}(\Gamma; \mathbb R))\] 
	respectively. Again, the surjectivity of $  \theta_{\max}$, hence that of $\Theta_{\max}$,  follows from the surjectivity of the maximal Baum-Connes assembly map 
	\[ \mu_{\mathbb R}\colon   KO_\bullet^\Gamma(\underline{E}\Gamma)\to KO_\bullet(C_{\max}^*(\Gamma;\mathbb R)). \] 
}


\subsection{Stable bounding with respect to $B\Gamma$}
	In this subsection, we shall discuss how the assumption that a multiple of $(M, \varphi, h)$ stably bounds with respect to $B\Gamma$ (cf. Definition $\ref{def:bound}$) is related to the Baum-Connes conjecture and the Stolz conjecture. Here again  $M$ is  a closed spin manifold equipped with a Riemannian metric $h$ of positive scalar curvature. Let $\varphi\colon M \to B\Gamma$ be the classifying map for the covering $\widetilde M \to M$, that is, the pullback of $E\Gamma$ by $\varphi$ is $\widetilde M$. 
	
	To be more precise, in this subsection, let us assume the Baum-Connes assembly map 
	\[ \mu_{\mathbb R}\colon   KO_\bullet^\Gamma(\underline{E}\Gamma)\to KO_\bullet(C_{r}^*(\Gamma;\mathbb R)) \] is rationally isomorphic\footnote{The rational bijectivity of $\mu_{\mathbb R}$ follows from the rational bijectivity of the complex version $\mu\colon  K_\bullet^\Gamma(\underline{E}\Gamma)\to K_\bullet(C_{r}^*(\Gamma))$, cf. \cite{PBMK04}. }. There is a long exact sequence for $KO$-theory of reduced $C^\ast$-algebras analogous to commutative diagram $\eqref{cd:longexact}$. Now a similar argument as in the proof of Theorem $\ref{prop:max}$ shows the higher rho invariant $\rho(\widetilde D) = \partial [p] $ (up to rational multiples) for some element $[p]\in KO_{n+1}(C^\ast_r(\Gamma; \mathbb R))$, where $n = \dim M$. Now the rational surjectivity of the Baum-Connes assembly map 
	\[ \mu_{\mathbb R}\colon   KO_\bullet^\Gamma(\underline{E}\Gamma)\to KO_\bullet(C_{r}^*(\Gamma;\mathbb R))\] implies the   rational surjectivity of the Stolz map \[ \theta\colon  R^{\spin}_\bullet (B\Gamma)[\mathfrak B^{-1}] \to KO_\bullet(C^\ast_r(\Gamma; \mathbb R),\] cf.  \cite[Corollary 3.15]{Xie:2017aa}. And  the rational surjectivity of $\theta$ implies that there exists an element $(L, f, h) \in R^{\spin}_{n+1} (B\Gamma)[\mathfrak B^{-1}] $ such that $\theta(L, f, h)  = [p]$ (up to a rational mutiple). Recall that (cf. \cite[Theorem 1.14]{MR3286895}\cite[Theorem A]{Xiepos})
	\[   \partial (\theta(L, f, g)) = \rho(D_{\widetilde {\partial L}}) \textup{ in } KO_n(C_{L,0}^\ast(E\Gamma; \mathbb R)^\Gamma), \]
	where $\rho(D_{\widetilde {\partial L}})$ is the higher rho invariant of $D_{\widetilde{\partial L}}$ with respect to the positive scalar curvature metric $h$. In particular, this implies that $\rho(\widetilde D) = \rho(D_{\widetilde {\partial L}})$. Hence, as far as $\rho(\widetilde D)$ is concerned,   we could work with  $(\partial L, f, g)$, which clearly bounds,  instead of  $(M, \varphi, h)$. On the other hand,  it is an open question whether the higher rho invariants for $M$ and $\partial L$ remain equal to each other, for corresponding finite-sheeted covering spaces of $M$ and $\partial L$. 

\subsection{Positively stable bounding with respect to $\underline{E}\Gamma$}

In this subsection, we shall discuss how the assumption that a multiple of $(\widetilde M, \widetilde h)$ positively stably  bounds with respect to $\underline{E}\Gamma$ (cf. Definition $\ref{def:properbound}$) is related to the Baum-Connes conjecture and the generalized Stolz conjecture.

Observe that the injectivity of the generalized Stolz map
\[ \Theta\colon  R^{\spin}_n(\underline{E}\Gamma)^\Gamma[\mathfrak B^{-1}] \to KO_n(C^\ast_r(\Gamma; \mathbb R))\] 
  has the following  immediate geometric consequence. Let $(L, f, h)$ be an element in $ R^{\spin}_n(\underline{E}\Gamma)^\Gamma$, that is, $L$ is a $n$-dimensional spin $\Gamma$-manifold whose boundary $\partial L$ carries a $\Gamma$-invariant positive scalar curvature metric $h$, together with a $\Gamma$-equivariant map $f\colon L \to \underline{E}\Gamma$.   Suppose the higher index $\ind(D_{L})$ associated to $(L, f, h)$ vanishes, then the injectivity of the generalized Stolz map implies   that $(L, f, h)$ is stably $\Gamma$-equivariantly cobordant to the empty set. More precisely,   if $(L', f', h')$ is the direct product of $(L, f, h)$ with sufficiently many copies of $\mathfrak B$, then $(L', f', h')$ is $\Gamma$-equivariantly cobordant to the empty set. In particular, this implies that, if the higher index of $(L, f, h)$ vanishes, then $\partial L $ positively stably  bounds with respect to $\underline{E}\Gamma$, cf. Definition $\ref{def:properbound}$. More precisely,  $\partial L'$ bounds a spin $\Gamma$-manifold $V$  such that $V$ admits a $\Gamma$-invariant positive scalar curvature metric $g_0$,  which has product structure near the boundary $\partial V = \partial L'$,  and the restriction of $g_0$ to the boundary is equal to $h'$.

Now let $M$ be a closed spin manifold equipped with a Riemannian metric $h$ of positive scalar curvature. Let $\varphi\colon M \to B\Gamma$ be the classifying map for the covering $\widetilde M \to M$ and $\tilde h$ the metric on $\widetilde M$ lifted from $h$. The above discussion has the following consequence.

\begin{lemma}
With the above notation, suppose a multiple of  $(M, 
\varphi, h)$ stably bounds with respect to $B\Gamma$. If   the Baum-Connes assembly map 
\[ \mu_{\mathbb R}\colon   KO_\bullet^\Gamma(\underline{E}\Gamma)\to KO_\bullet(C_{r}^*(\Gamma;\mathbb R)) \] is rationally surjective and the generalized Stolz map  
\[ \Theta\colon  R^{\spin}_n(\underline{E}\Gamma)^\Gamma[\mathfrak B^{-1}] \to KO_n(C^\ast_r(\Gamma; \mathbb R))\]  is rationally injective, then $(\widetilde M, \widetilde h)$ positively stably  bounds with respect to $\underline{E}\Gamma$.
\end{lemma}
\begin{proof}
	For notational simplicity, let us assume $(M, \varphi, h)$ itself bounds with respect to $B\Gamma$,  that is, there exists a compact spin manifold $W$ and a map $\Phi\colon W\to B\Gamma$ such that $\partial W = M$ and  $\Phi|_{\partial W} = \varphi$. 
	
	Endow $W$ with a Riemannian metric $g$ which has product structure near $\partial W = M$ and whose restriction on $\partial W$ is the positive scalar curvature metric $h$. Let $\widetilde W$ be the covering space of $W$ induced by the map $\Phi\colon W\to B\Gamma$ and $\tilde g$ be the lift of $g$ from $W$ to $\widetilde W$. Due to the positive scalar curvature of $\tilde g$ near the boundary of $\widetilde W$,  the corresponding Dirac operator $D_{\widetilde W}$ on $\widetilde W$ with respect to the metric $\tilde g$ has a well-defined higher index $\ind(D_{\widetilde W}, \tilde g)$ in $KO_{n+1}(C_r^\ast(\Gamma;\mathbb R)). $ 
	
	By the (rational) surjectivity of  the Baum-Connes assembly map 
	\[ \mu_{\mathbb R}\colon   KO_\bullet^\Gamma(\underline{E}\Gamma)\to KO_\bullet(C_{r}^*(\Gamma;\mathbb R)), \] there exists a spin $\Gamma$-manifold $Z$ (without boundary) such that the higher index $\ind(D_Z)$ of its Dirac operator $D_Z$ is  equal to $-\ind(D_{\widetilde W}, \tilde g)$ (up to a rational multiple). Let $Z_1$ be the $\Gamma$-equivariant connected sum\footnote{The connected sum is performed away from the boundary of $\widetilde W$.} of $\widetilde W$ with $Z$. Then $Z_1$ is a spin $\Gamma$-manifold whose boundary is equal to $\partial \widetilde W = \widetilde M$. Moreover, the higher index $\ind(D_{Z_1})$ of the Dirac operator $D_{Z_1}$ is zero. Now it follows from the discussion in this subsection that
 $(\widetilde M, \tilde h) = (\partial Z_1, \tilde h)$ positively stably bounds with respect to $\underline{E}\Gamma$. 
	
\end{proof}


\begin{thebibliography}{10}
	
	\bibitem{A-P-S75a}
	M.~F. Atiyah, V.~K. Patodi, and I.~M. Singer.
	\newblock Spectral asymmetry and {R}iemannian geometry. {I}.
	\newblock {\em Math. Proc. Cambridge Philos. Soc.}, 77:43--69, 1975.
	
	\bibitem{A-P-S75b}
	M.~F. Atiyah, V.~K. Patodi, and I.~M. Singer.
	\newblock Spectral asymmetry and {R}iemannian geometry. {II}.
	\newblock {\em Math. Proc. Cambridge Philos. Soc.}, 78(3):405--432, 1975.
	
	\bibitem{A-P-S76}
	M.~F. Atiyah, V.~K. Patodi, and I.~M. Singer.
	\newblock Spectral asymmetry and {R}iemannian geometry. {III}.
	\newblock {\em Math. Proc. Cambridge Philos. Soc.}, 79(1):71--99, 1976.
	
	\bibitem{PBAC88}
	Paul Baum and Alain Connes.
	\newblock {$K$}-theory for discrete groups.
	\newblock In {\em Operator algebras and applications, {V}ol.\ 1}, volume 135 of
	{\em London Math. Soc. Lecture Note Ser.}, pages 1--20. Cambridge Univ.
	Press, Cambridge, 1988.
	
	\bibitem{BaumConnesHigson}
	Paul Baum, Alain Connes, and Nigel Higson.
	\newblock Classifying space for proper actions and {$K$}-theory of group
	{$C^\ast$}-algebras.
	\newblock In {\em {$C^\ast$}-algebras: 1943--1993 ({S}an {A}ntonio, {TX},
		1993)}, volume 167 of {\em Contemp. Math.}, pages 240--291. Amer. Math. Soc.,
	Providence, RI, 1994.
	
	\bibitem{MR679698}
	Paul Baum and Ronald~G. Douglas.
	\newblock {$K$}\ homology and index theory.
	\newblock In {\em Operator algebras and applications, {P}art {I} ({K}ingston,
		{O}nt., 1980)}, volume~38 of {\em Proc. Sympos. Pure Math.}, pages 117--173.
	Amer. Math. Soc., Providence, R.I., 1982.
	
	\bibitem{PBMK04}
	Paul Baum and Max Karoubi.
	\newblock On the {B}aum-{C}onnes conjecture in the real case.
	\newblock {\em Q. J. Math.}, 55(3):231--235, 2004.
	
	\bibitem{MR1388307}
	M.~E.~B. Bekka, P.-A. Cherix, and A.~Valette.
	\newblock Proper affine isometric actions of amenable groups.
	\newblock In {\em Novikov conjectures, index theorems and rigidity, {V}ol. 2
		({O}berwolfach, 1993)}, volume 227 of {\em London Math. Soc. Lecture Note
		Ser.}, pages 1--4. Cambridge Univ. Press, Cambridge, 1995.
	
	\bibitem{MR1339924}
	Boris Botvinnik and Peter~B. Gilkey.
	\newblock The eta invariant and metrics of positive scalar curvature.
	\newblock {\em Math. Ann.}, 302(3):507--517, 1995.
	
	\bibitem{UB95}
	Ulrich Bunke.
	\newblock A {$K$}-theoretic relative index theorem and {C}allias-type {D}irac
	operators.
	\newblock {\em Math. Ann.}, 303(2):241--279, 1995.
	
	\bibitem{MR4051922}
	Xiaoman Chen, Hongzhi Liu, and Guoliang Yu.
	\newblock Higher {$\rho$} invariant is an obstruction to the inverse being
	local.
	\newblock {\em J. Geom. Phys.}, 150:103592, 18, 2020.
	
	\bibitem{CWXY}
	Xiaoman Chen, Jinmin Wang, Zhizhang Xie, and Guoliang Yu.
	\newblock Delocalized eta invariants, cyclic cohomology and higher rho
	invariants.
	\newblock {\em arXiv:1901.02378}, 2019.
	
	\bibitem{CM90}
	Alain Connes and Henri Moscovici.
	\newblock Cyclic cohomology, the {N}ovikov conjecture and hyperbolic groups.
	\newblock {\em Topology}, 29(3):345--388, 1990.
	
	\bibitem{MR511246}
	Harold Donnelly.
	\newblock Eta invariants for {$G$}-spaces.
	\newblock {\em Indiana Univ. Math. J.}, 27(6):889--918, 1978.
	
	\bibitem{MR2431253}
	Guihua Gong, Qin Wang, and Guoliang Yu.
	\newblock Geometrization of the strong {N}ovikov conjecture for residually
	finite groups.
	\newblock {\em J. Reine Angew. Math.}, 621:159--189, 2008.
	
	\bibitem{Guo:2019aa}
	Hao Guo, Zhizhang Xie, and Guoliang Yu.
	\newblock A {L}ichnerowicz vanishing theorem for the maximal {R}oe algebra.
	\newblock arXiv:1905.12299, 2019.
	
	\bibitem{guoxieyu}
	Hao Guo, Zhizhang Xie, and Guoliang Yu.
	\newblock Functoriality for higher rho invariants of elliptic operators.
	\newblock {\em J. Funct. Anal.}, 280(10):108966, 36, 2021.
	
	\bibitem{MR1821144}
	Nigel Higson and Gennadi Kasparov.
	\newblock {$E$}-theory and {$KK$}-theory for groups which act properly and
	isometrically on {H}ilbert space.
	\newblock {\em Invent. Math.}, 144(1):23--74, 2001.
	
	\bibitem{MR1383960}
	D.~D. Joyce.
	\newblock Compact {$8$}-manifolds with holonomy {${\rm Spin}(7)$}.
	\newblock {\em Invent. Math.}, 123(3):507--552, 1996.
	
	\bibitem{MR2874956}
	Vincent Lafforgue.
	\newblock La conjecture de {B}aum-{C}onnes \`a coefficients pour les groupes
	hyperboliques.
	\newblock {\em J. Noncommut. Geom.}, 6(1):1--197, 2012.
	
	\bibitem{LeichtnamPos}
	Eric Leichtnam and Paolo Piazza.
	\newblock On higher eta-invariants and metrics of positive scalar curvature.
	\newblock {\em $K$-Theory}, 24(4):341--359, 2001.
	
	\bibitem{Lott}
	John Lott.
	\newblock Delocalized {$L^2$}-invariants.
	\newblock {\em J. Funct. Anal.}, 169(1):1--31, 1999.
	
	\bibitem{MR1914618}
	Igor Mineyev and Guoliang Yu.
	\newblock The {B}aum-{C}onnes conjecture for hyperbolic groups.
	\newblock {\em Invent. Math.}, 149(1):97--122, 2002.
	
	\bibitem{MR3286895}
	Paolo Piazza and Thomas Schick.
	\newblock Rho-classes, index theory and {S}tolz' positive scalar curvature
	sequence.
	\newblock {\em J. Topol.}, 7(4):965--1004, 2014.
	
	\bibitem{Puschnigg}
	Michael Puschnigg.
	\newblock New holomorphically closed subalgebras of {$C^*$}-algebras of
	hyperbolic groups.
	\newblock {\em Geom. Funct. Anal.}, 20(1):243--259, 2010.
	
	\bibitem{Roe}
	John Roe.
	\newblock {\em Index theory, coarse geometry, and topology of manifolds},
	volume~90 of {\em CBMS Regional Conference Series in Mathematics}.
	\newblock Published for the Conference Board of the Mathematical Sciences,
	Washington, DC; by the American Mathematical Society, Providence, RI, 1996.
	
	\bibitem{MR3439130}
	John Roe.
	\newblock Positive curvature, partial vanishing theorems and coarse indices.
	\newblock {\em Proc. Edinb. Math. Soc. (2)}, 59(1):223--233, 2016.
	
	\bibitem{Song:2019aa}
	Yanli Song and Xiang Tang.
	\newblock Higher orbit integrals, cyclic cocyles, and {K}-theory of reduced
	group {C*}-algebra.
	\newblock arXiv:1910.00175, 2019.
	
	\bibitem{Stebe}
	Peter~F. Stebe.
	\newblock Conjugacy separability of groups of integer matrices.
	\newblock {\em Proc. Amer. Math. Soc.}, 32:1--7, 1972.
	
	\bibitem{stolz-concordance}
	Stephan Stolz.
	\newblock Concordance classes of positive scalar curvature metrics.
	\newblock preprint available at \url{http://www.nd.edu/~stolz}.
	
	\bibitem{SS95}
	Stephan Stolz.
	\newblock Positive scalar curvature metrics---existence and classification
	questions.
	\newblock In {\em Proceedings of the {I}nternational {C}ongress of
		{M}athematicians, {V}ol.\ 1, 2 ({Z}{\"u}rich, 1994)}, pages 625--636, Basel,
	1995. Birkh{\"a}user.
	
	\bibitem{MR3454548}
	Bai-Ling Wang and Hang Wang.
	\newblock Localized index and {$L^2$}-{L}efschetz fixed-point formula for
	orbifolds.
	\newblock {\em J. Differential Geom.}, 102(2):285--349, 2016.
	
	\bibitem{Weinberger:2016dq}
	Shmuel Weinberger, Zhizhang Xie, and Guoliang Yu.
	\newblock Additivity of higher rho invariants and nonrigidity of topological
	manifolds.
	\newblock {\em Comm. Pure Appl. Math.}, 74(1):3--113, 2021.
	
	\bibitem{Xiepos}
	Zhizhang Xie and Guoliang Yu.
	\newblock Positive scalar curvature, higher rho invariants and localization
	algebras.
	\newblock {\em Adv. Math.}, 262:823--866, 2014.
	
	\bibitem{MR3122162}
	Zhizhang Xie and Guoliang Yu.
	\newblock A relative higher index theorem, diffeomorphisms and positive scalar
	curvature.
	\newblock {\em Adv. Math.}, 250:35--73, 2014.
	
	\bibitem{MR3590536}
	Zhizhang Xie and Guoliang Yu.
	\newblock Higher rho invariants and the moduli space of positive scalar
	curvature metrics.
	\newblock {\em Adv. Math.}, 307:1046--1069, 2017.
	
	\bibitem{Xie}
	Zhizhang Xie and Guoliang Yu.
	\newblock Delocalized eta invariants, algebraicity, and {$K$}-theory of group
	{$C^*$}-algebras.
	\newblock {\em International Mathematics Research Notices}, 2019.
	
	\bibitem{Xie:2019aa}
	Zhizhang Xie and Guoliang Yu.
	\newblock Higher invariants in noncommutative geometry.
	\newblock In {\em Advances in Noncommutative Geometry - On the Occasion of
		Alain Connes' 70th Birthday}. Springer, 2019.
	
	\bibitem{Xie:2017aa}
	Zhizhang Xie, Guoliang Yu, and Rudolf Zeidler.
	\newblock On the range of the relative higher index and the higher
	rho-invariant for positive scalar curvature.
	\newblock {\em arXiv:1712.03722}, 2017.
	
	\bibitem{Yulocalization}
	Guoliang Yu.
	\newblock Localization algebras and the coarse {B}aum-{C}onnes conjecture.
	\newblock {\em $K$-Theory}, 11(4):307--318, 1997.
	
	\bibitem{YuChar}
	Guoliang Yu.
	\newblock A characterization of the image of the {B}aum-{C}onnes map.
	\newblock In {\em Quanta of maths}, volume~11 of {\em Clay Math. Proc.}, pages
	649--657. Amer. Math. Soc., Providence, RI, 2010.
	
	\bibitem{MR3551834}
	Rudolf Zeidler.
	\newblock Positive scalar curvature and product formulas for secondary index
	invariants.
	\newblock {\em J. Topol.}, 9(3):687--724, 2016.
	
\end{thebibliography}
\end{document}